\documentclass[10pt,portrait,reqno,12pt]{amsart}%
\usepackage{amssymb,amsthm,amsmath, amsfonts, float}
\usepackage{graphicx}
\usepackage{multirow}
\usepackage{subfig}
\usepackage{amsmath}
\usepackage{amsfonts}
\usepackage{lscape}
\usepackage{amssymb}%
\providecommand{\U}[1]{\protect\rule{.1in}{.1in}}
\newtheorem{theorem}{Theorem}
\theoremstyle{plain}

\newtheorem{proposition}{Proposition}
\newtheorem{remark}{Remark}

\numberwithin{equation}{section}
\begin{document}
\title[Canal Hypersurfaces in $E_{1}^{4}$]{Canal Hypersurfaces Generated by Non-Null Curves in Lorentz-Minkowski 4-Space}
\author[A. Kazan, M. Alt\i n and D.W. Yoon]{\bfseries Ahmet Kazan$^{1}$, Mustafa Alt\i n$^{2}$ and Dae Won Yoon$^{3}$}
\address{ \\
$^{1}$Department of Computer Technologies, Do\u{g}an\c{s}ehir Vahap
K\"{u}\c{c}\"{u}k Vocational School, Malatya Turgut \"{O}zal University,
Malatya, Turkey \\
 \\
$^{2}$Technical Sciences Vocational School, Bing\"{o}l University, Bing\"{o}l,
Turkey \\
 \\
$^{3}$Department of Mathematics Education and RINS, Gyeongsang National
University, Jinju 52828, Republic of Korea \\
 \\
$^{\ast}$Corresponding author: ahmet.kazan@ozal.edu.tr}

\begin{abstract}
In the present paper, firstly we obtain the general expression of the canal
hypersurfaces which are formed as the envelope of a family of pseudo
hyperspheres, pseudo hyperbolic hyperspheres and null hypercones in $E_{1}%
^{4}$ and give their some geometric invariants such as unit normal vector
fields, Gaussian curvatures, mean curvatures and principal curvatures. Also we
give some results about their flatness and minimality conditions and
Weingarten canal hypersurfaces. Also, we obtain these characterizations for
tubular hypersurfaces in $E_{1}^{4}$ by taking constant radius function and
finally we construct some examples and visualize them with the aid of Mathematica.

\end{abstract}
\maketitle


\section{\textbf{GENERAL INFORMATION AND\ BASIC\ CONCEPTS}}

Canal surfaces, which are the class of surfaces formed by sweeping a sphere,
have been investigated by Monge in 1850. Thus, a canal surface can be seen as
the envelope of a moving sphere with varying radius, defined by the trajectory
$\beta(s)$ of its centers and a radius function $r(s)$. If the radius function
is constant, then the canal surface is called a tubular surface or pipe
surface. Canal surfaces (especially tubular surfaces) have been applied to
many fields, such as the solid and the surface modeling for CAD/CAM,
construction of blending surfaces, shape re-construction and they are useful
to represent various objects such as pipe, hose, rope or intestine of a body
\cite{Karacany}, \cite{Ucum}. In this context, canal and tubular
(hyper)surfaces have been studied by many mathematicians in different three,
four or higher dimensional spaces (see \cite{Mahmut}, \cite{Aslan},
\cite{Yusuf}, \cite{Fu}, \cite{Garcia}, \cite{Hartman}, \cite{Izumiya},
\cite{Karacan}, \cite{Karacan2}, \cite{Karacan3}, \cite{Sezai}, \cite{Kim},
\cite{Krivos}, \cite{Maekawa}, \cite{Peter}, \cite{Ro}, \cite{Ucum},
\cite{Xu}, \cite{Kucuk} and etc.).

Since the extrinsic differential geometry of submanifolds in Lorentz-Minkowski
4-space $E_{1}^{4}$ is of special interest in Relativity Theory, lots of
studies about curves and (hyper)surfaces have been done in this space and this
motivated us to construct the canal hypersurfaces in $E_{1}^{4}$. Now, let us
recall some basic concepts for curves and hypersurfaces in $E_{1}^{4}$.

If $\overrightarrow{x}=(x_{1},x_{2},x_{3},x_{4})$, $\overrightarrow{y}%
=(y_{1},y_{2},y_{3},y_{4})$ and $\overrightarrow{z}=(z_{1},z_{2},z_{3},z_{4})$
are three vectors in $E_{1}^{4}$, then the inner product and vector product
are defined by%
\begin{equation}
\left\langle \overrightarrow{x},\overrightarrow{y}\right\rangle =-x_{1}%
y_{1}+x_{2}y_{2}+x_{3}y_{3}+x_{4}y_{4} \label{yy1}%
\end{equation}
and
\begin{equation}
\overrightarrow{x}\times\overrightarrow{y}\times\overrightarrow{z}=\det\left[
\begin{array}
[c]{cccc}%
-e_{1} & e_{2} & e_{3} & e_{4}\\
x_{1} & x_{2} & x_{3} & x_{4}\\
y_{1} & y_{2} & y_{3} & y_{4}\\
z_{1} & z_{2} & z_{3} & z_{4}%
\end{array}
\right]  , \label{yy2}%
\end{equation}
respectively.

A vector $\overrightarrow{x}\in E_{1}^{4}-\{0\}$ is called spacelike if
$\left\langle \overrightarrow{x},\overrightarrow{x}\right\rangle >0$; timelike
if $\left\langle \overrightarrow{x},\overrightarrow{x}\right\rangle <0$ and
lightlike (null) if $\left\langle \overrightarrow{x},\overrightarrow{x}%
\right\rangle =0$. In particular, the vector $\overrightarrow{x}=0$ is
spacelike. Also, the norm of the vector $\overrightarrow{x}$ is $\left\Vert
\overrightarrow{x}\right\Vert =\sqrt{\left\vert \left\langle
\overrightarrow{x},\overrightarrow{x}\right\rangle \right\vert }$. A curve
$\beta(s)$ in $E_{1}^{4}$ is spacelike, timelike or lightlike (null), if all
its velocity vectors $\beta^{\prime}(s)$ are spacelike, timelike or lightlike,
respectively and a non-null (i.e. timelike or spacelike) curve has unit speed
if $\left\langle \beta^{\prime},\beta^{\prime}\right\rangle =\mp1$
\cite{Kuhnel}.

If $\{F_{1},F_{2},F_{3},F_{4}\}$ is the moving Frenet frame along the timelike
or spacelike curve $\beta(s)$ in $E_{1}^{4},$ where we'll call $F_{1},$
$F_{2},$ $F_{3}$ and $F_{4}$ are unit tangent vector field, principal normal
vector field, binormal vector field and trinormal vector field, respectively,
then the Frenet equations can be given according to the causal characters of
non-null Frenet vector fields $F_{1},$ $F_{2},$ $F_{3}$ and $F_{4}$ as follows
\cite{Walrave}:

If the curve $\beta(s)$ is timelike, i.e. $\left\langle F_{1},F_{1}%
\right\rangle =-1,$ $\left\langle F_{2},F_{2}\right\rangle =\left\langle
F_{3},F_{3}\right\rangle =\left\langle F_{4},F_{4}\right\rangle =1,$ then%
\begin{equation}
\left[
\begin{array}
[c]{c}%
F_{1}^{\prime}\\
F_{2}^{\prime}\\
F_{3}^{\prime}\\
F_{4}^{\prime}%
\end{array}
\right]  =\left[
\begin{array}
[c]{cccc}%
0 & k_{1} & 0 & 0\\
k_{1} & 0 & k_{2} & 0\\
0 & -k_{2} & 0 & k_{3}\\
0 & 0 & -k_{3} & 0
\end{array}
\right]  \left[
\begin{array}
[c]{c}%
F_{1}\\
F_{2}\\
F_{3}\\
F_{4}%
\end{array}
\right]  ; \label{fr1}%
\end{equation}
if the curve $\beta(s)$ is spacelike with timelike principal normal vector
field $F_{2}$, i.e. $\left\langle F_{2},F_{2}\right\rangle =-1,$ $\left\langle
F_{1},F_{1}\right\rangle =\left\langle F_{3},F_{3}\right\rangle =\left\langle
F_{4},F_{4}\right\rangle =1,$ then%
\begin{equation}
\left[
\begin{array}
[c]{c}%
F_{1}^{\prime}\\
F_{2}^{\prime}\\
F_{3}^{\prime}\\
F_{4}^{\prime}%
\end{array}
\right]  =\left[
\begin{array}
[c]{cccc}%
0 & k_{1} & 0 & 0\\
k_{1} & 0 & k_{2} & 0\\
0 & k_{2} & 0 & k_{3}\\
0 & 0 & -k_{3} & 0
\end{array}
\right]  \left[
\begin{array}
[c]{c}%
F_{1}\\
F_{2}\\
F_{3}\\
F_{4}%
\end{array}
\right]  ; \label{fr2}%
\end{equation}
if the curve $\beta(s)$ is spacelike with timelike binormal vector field
$F_{3}$, i.e. $\left\langle F_{3},F_{3}\right\rangle =-1,$ $\left\langle
F_{1},F_{1}\right\rangle =\left\langle F_{2},F_{2}\right\rangle =\left\langle
F_{4},F_{4}\right\rangle =1,$ then%
\begin{equation}
\left[
\begin{array}
[c]{c}%
F_{1}^{\prime}\\
F_{2}^{\prime}\\
F_{3}^{\prime}\\
F_{4}^{\prime}%
\end{array}
\right]  =\left[
\begin{array}
[c]{cccc}%
0 & k_{1} & 0 & 0\\
-k_{1} & 0 & k_{2} & 0\\
0 & k_{2} & 0 & k_{3}\\
0 & 0 & k_{3} & 0
\end{array}
\right]  \left[
\begin{array}
[c]{c}%
F_{1}\\
F_{2}\\
F_{3}\\
F_{4}%
\end{array}
\right]  ; \label{fr3}%
\end{equation}
if the curve $\beta(s)$ is spacelike with timelike trinormal vector field
$F_{4}$, i.e. $\left\langle F_{4},F_{4}\right\rangle =-1,$ $\left\langle
F_{1},F_{1}\right\rangle =\left\langle F_{2},F_{2}\right\rangle =\left\langle
F_{3},F_{3}\right\rangle =1,$ then%
\begin{equation}
\left[
\begin{array}
[c]{c}%
F_{1}^{\prime}\\
F_{2}^{\prime}\\
F_{3}^{\prime}\\
F_{4}^{\prime}%
\end{array}
\right]  =\left[
\begin{array}
[c]{cccc}%
0 & k_{1} & 0 & 0\\
-k_{1} & 0 & k_{2} & 0\\
0 & -k_{2} & 0 & k_{3}\\
0 & 0 & k_{3} & 0
\end{array}
\right]  \left[
\begin{array}
[c]{c}%
F_{1}\\
F_{2}\\
F_{3}\\
F_{4}%
\end{array}
\right]  , \label{fr4}%
\end{equation}
where $k_{1},k_{2},k_{3}$ are the first, second and third curvatures of the
non-null curve $\beta(s).$

Furthermore, if $p$ is a fixed point in $E_{1}^{4}$ and $r$ is a positive
constant, then the pseudo-Riemannian hypersphere is defined by%
\[
S_{1}^{3}(p,r)=\{x\in E_{1}^{4}:\left\langle x-p,x-p\right\rangle =r^{2}\},
\]
the pseudo-Riemannian hyperbolic space is defined by
\[
H_{0}^{3}(p,r)=\{x\in E_{1}^{4}:\left\langle x-p,x-p\right\rangle =-r^{2}\},
\]
the pseudo-Riemannian null hypercone is defined by
\[
Q_{1}^{3}=\{x\in E_{1}^{4}:\left\langle x-p,x-p\right\rangle =0\}.
\]

In the present study, we construct the canal hypersurfaces in $E_{1}^{4}$ as
the envelope of a family of pseudo hyperspheres, pseudo hyperbolic
hyperspheres or null hypercones whose centers lie on a non-null space curve.

On the other hand, the differential geometry of different types of
(hyper)surfaces in 4-dimensional spaces has been a popular topic for
geometers, recently (\cite{Altin4}, \cite{Altin2}, \cite{Altin3},
\cite{Aydin}, \cite{Altin}, \cite{Kisi} and etc). In this context, let
$\Gamma$ be a hypersurface in $E_{1}^{4}$ given by
\begin{align}
\Gamma:U\subset E^{3}  &  \longrightarrow E_{1}^{4}\label{yy3}\\
(u_{1},u_{2},u_{3})  &  \longrightarrow\Gamma(u_{1},u_{2},u_{3})=(\Gamma
_{1}(u_{1},u_{2},u_{3}),\Gamma_{2}(u_{1},u_{2},u_{3}),\Gamma_{3}(u_{1}%
,u_{2},u_{3}),\Gamma_{4}(u_{1},u_{2},u_{3})).\nonumber
\end{align}
Then the Gauss map (i.e., the unit normal vector field), the matrix forms of
the first and second fundamental forms are%
\begin{equation}
N_{\Gamma}=\frac{\Gamma_{u_{1}}\times\Gamma_{u_{2}}\times\Gamma_{u_{3}}%
}{\left\Vert \Gamma_{u_{1}}\times\Gamma_{u_{2}}\times\Gamma_{u_{3}}\right\Vert
}, \label{4y}%
\end{equation}%
\begin{equation}
\lbrack g_{ij}]=\left[
\begin{array}
[c]{ccc}%
g_{11} & g_{12} & g_{13}\\
g_{21} & g_{22} & g_{23}\\
g_{31} & g_{32} & g_{33}%
\end{array}
\right]  \label{5y}%
\end{equation}
and%
\begin{equation}
\lbrack h_{ij}]=\left[
\begin{array}
[c]{ccc}%
h_{11} & h_{12} & h_{13}\\
h_{21} & h_{22} & h_{23}\\
h_{31} & h_{32} & h_{33}%
\end{array}
\right]  , \label{6y}%
\end{equation}
respectively. Here $g_{ij}=\left\langle \Gamma_{u_{i}},\Gamma_{u_{j}%
}\right\rangle ,$ $h_{ij}=\left\langle \Gamma_{u_{i}u_{j}},N_{\Gamma
}\right\rangle ,$ $\Gamma_{u_{i}}=\frac{\partial\Gamma}{\partial u_{i}},$
$\Gamma_{u_{i}u_{j}}=\frac{\partial^{2}\Gamma}{\partial u_{i}u_{j}},$
$i,j\in\{1,2,3\}.$

Also, the matrix of shape operator of the hypersurface (\ref{yy3}) is%
\begin{equation}
S=[a_{ij}]=[g^{ij}].[h_{ij}], \label{7yyy}%
\end{equation}
where $[g^{ij}]$ is the inverse matrix of $[g_{ij}]$.

With the aid of (\ref{5y})-(\ref{7yyy}), the Gaussian curvature and mean
curvature of a hypersurface in $E_{1}^{4}$ are given by%
\begin{equation}
K=\varepsilon\frac{\det[h_{ij}]}{\det[g_{ij}]} \label{yy4}%
\end{equation}
and%
\begin{equation}
3\varepsilon H=tr(S), \label{yy5}%
\end{equation}
respectively. Here, $\varepsilon=\left\langle N_{\Gamma},N_{\Gamma
}\right\rangle .$ We say that a hypersurface is flat or minimal, if it has
zero Gaussian or zero mean curvature, respectively. For more details about
hypersurfaces in $E_{1}^{4},$ we refer to \cite{Guler1}, \cite{Lee} and etc.

In the second section of this paper, we obtain the general expression of the
canal hypersurfaces which are formed as the envelope of a family of pseudo
hyperspheres, pseudo hyperbolic hyperspheres and null hypercones in $E_{1}%
^{4}$ and give some characterizations for them. In the third section, we
obtain these results for tubular hypersurfaces by taking constant radius
function and in the last section, we construct some examples and visualize
them with the aid of Mathematica.

\section{\textbf{CANAL HYPERSURFACES GENERATED BY NON-NULL CURVES IN }%
$E_{1}^{4}$}

In this section, firstly we construct the canal hypersurfaces which are formed
as the envelope of a family of pseudo hyperspheres, pseudo hyperbolic
hyperspheres and null hypercones in $E_{1}^{4}$. After that, we obtain some
important geometric invariants such as unit normal vector fields, Gaussian
curvatures, mean curvatures and principal curvatures of these canal
hypersurfaces in general form. Also, we give some results for flat, minimal
and Weingarten canal hypersurfaces in $E_{1}^{4}$.

\subsection{\textbf{CONSTRUCTION OF CANAL HYPERSURFACES}}

\

Here, we prove a theorem which gives us the general parametric expressions of
11 different types of canal hypersurfaces obtained by pseudo hyperspheres,
pseudo hyperbolic hyperspheres and null hypercones in $E_{1}^{4}$. Also, we
write the parametric expressions of these canal hypersurfaces explicitly.

\begin{theorem}
\label{teo1}\textit{The canal hypersurfaces which are formed as the envelope
of a family of pseudo hyperspheres or pseudo hyperbolic hyperspheres in
}$E_{1}^{4}$\textit{ generated by spacelike or timelike center curves} with
non-null Frenet vector fields \textit{can be parametrized by}%
\begin{equation}
\mathfrak{C}^{\{j;\lambda\}}{\small (s,t,w)=\beta(s)-\lambda\varepsilon}%
_{1}{\small r(s)r}^{\prime}{\small (s)F}_{1}{\small (s)\mp r(s)}%
\sqrt{{\small r}^{\prime}{\small (s)}^{2}{\small -\lambda\varepsilon}_{1}%
}\left(
{\textstyle\sum\nolimits_{i=2}^{4}}
\mathfrak{a}_{i}{\small (s,t,w)F}_{i}{\small (s)}\right)  {\small ,}%
\label{geneldenklemyy}%
\end{equation}
where%
\begin{align*}
&  \text{i) }g(F_{j},F_{j})=-1=\varepsilon_{j}\text{ \ and \ for }i\neq
j,\text{ }\varepsilon_{i}=1,\text{ }i,j\in\{1,2,3,4\};\\
& \\
&  \text{ii) \ }\left\{
\begin{array}
[c]{l}%
{\small j=1}\Rightarrow\mathfrak{a}_{2}{\small (s,t,w)=}\cos{\small t}%
\cos{\small w,}\text{ }\mathfrak{a}_{3}{\small (s,t,w)=}\sin{\small t}%
\cos{\small w,}\text{ }\mathfrak{a}_{4}{\small (s,t,w)=}\sin{\small w,}\\
\\
{\small j=2,3,4}\Rightarrow%
\begin{array}
[c]{l}%
\mathfrak{a}_{j}{\small (s,t,w)=}\cosh{\small t}\cosh{\small w,}\text{
}\mathfrak{a}_{j+1}{\small (s,t,w)=}\sinh{\small w,}\text{ }\mathfrak{a}%
_{j+2}{\small (s,t,w)=}\sinh{\small t}\cosh{\small w,}\\
\mathfrak{a}_{5}{\small (s,t,w)=}\mathfrak{a}_{2}{\small (s,t,w),}\text{
}\mathfrak{a}_{6}{\small (s,t,w)=}\mathfrak{a}_{3}{\small (s,t,w);}%
\end{array}
\end{array}
\right.
\end{align*}
also, we suppose $r^{\prime}(s)^{2}>\lambda\varepsilon_{1}$ and \textit{if the
canal hypersurface is foliated by pseudo hyperspheres or pseudo hyperbolic
hyperspheres, then }$\lambda=1$\textit{ or }$\lambda=-1,$\textit{
respectively.}

\textit{Furthermore, the canal hypersurfaces which are formed as the envelope
of a family of null hypercones can be parametrized by}%
\begin{equation}
{\small \mathfrak{C}}^{\{j;0\}}{\small (s,t,w)=\beta(s)\mp}\left(
{\textstyle\sum\nolimits_{i=2}^{4}}
a_{i}(s,t,w)F_{i}(s)\right)  {\small ,}\text{ \ }{\small j=2,3,4}
\label{geneldenklem2}%
\end{equation}
\textit{where }${\small a}_{i}{\small (s,t,w)}$ are arbitrary functions
satisfying\textit{ }%
\[%
{\textstyle\sum\nolimits_{i=2}^{4}}
{\small \varepsilon}_{i}{\small a}_{i}^{2}{\small (s,t,w)=0.}%
\]
\textit{It is obvious from this expression that, the canal hypersurface
}$\mathfrak{C}^{\{1;0\}}$\textit{ cannot be defined. }

\textit{Here, the center curve which generates the canal hypersurface is
timelike or spacelike if }$j=1$ or $j=2,3,4$\textit{, respectively. }
\end{theorem}

\begin{proof}
Let $\beta:I\subseteq%
\mathbb{R}
\rightarrow E_{1}^{4}$ be a spacelike or timelike center curve parametrized by
arc-length with non-zero curvature and $g(F_{j},F_{j})=\varepsilon_{j}=-1$,
where $F_{1}(s),$ $F_{2}(s),$ $F_{3}(s),$ $F_{4}(s)$ are unit tangent,
principal normal, binormal and trinormal vectors of $\beta(s)$, respectively.
Then, the parametrization of the envelope of pseudo hyperspheres ($\lambda=1$)
(resp. pseudo hyperbolic hyperspheres ($\lambda=-1$) or null hypercones
($\lambda=0$)) defining the canal hypersurfaces $\mathfrak{C}^{\{j;\lambda\}}$
in $E_{1}^{4}$ can be given by%
\begin{equation}
{\small \mathfrak{C}}^{\{j;\lambda\}}{\small (s,t,w)-\beta(s)=a}%
_{1}{\small (s,t,w)F}_{1}{\small (s)+a_{2}(s,t,w)F}_{2}{\small (s)+a_{3}%
(s,t,w)F}_{3}{\small (s)+a_{4}(s,t,w)F}_{4}{\small (s),} \label{1}%
\end{equation}
where $a_{i}(s,t,w)$ are differentiable functions of $s,t,w$ on the interval
$I$. Furthermore, since $\mathfrak{C}^{\{j;\lambda\}}(s,t,w)$ lies on the
pseudo hyperspheres (resp. pseudo hyperbolic hyperspheres or null hypercones),
we have%
\begin{equation}
{\small g(\mathfrak{C}}^{\{j;\lambda\}}{\small (s,t,w)-\beta(s),\mathfrak{C}%
}^{\{j;\lambda\}}{\small (s,t,w)-\beta(s))=\lambda r}^{2}{\small (s)}
\label{2}%
\end{equation}
which leads to from (\ref{1}) that%
\begin{equation}
{\small \varepsilon}_{1}{\small a}_{1}^{2}{\small +\varepsilon}_{2}%
{\small a}_{2}^{2}{\small +\varepsilon}_{3}{\small a}_{3}^{2}%
{\small +\varepsilon}_{4}{\small a}_{4}^{2}{\small =\lambda r}^{2} \label{3}%
\end{equation}
and%
\begin{equation}
{\small \varepsilon}_{1}{\small a}_{1}a_{1_{s}}{\small +\varepsilon}%
_{2}{\small a}_{2}a_{2_{s}}{\small +\varepsilon}_{3}{\small a}_{3}a_{3_{s}%
}{\small +\varepsilon}_{4}{\small a}_{4}a_{4_{s}}{\small =\lambda rr}%
_{s}{\small ,} \label{4}%
\end{equation}
where $r(s)$ is the radius function; $r=r(s),$ $r_{s}=\frac{dr(s)}{ds},$
$a_{i}=a_{i}(s,t,w),$ $a_{i_{s}}=\frac{\partial a_{i}(s,t,w)}{\partial s}$.

Also from (\ref{fr1})-(\ref{fr4}), the non-null Frenet vectors $F_{i}(s),$
$i\in\{1,2,3,4\},$ of the spacelike or timelike curve $\beta(s)$ satisfy the
relations%
\begin{equation}
\left.
\begin{array}
[c]{l}%
{\small F}_{1}^{\prime}{\small (s)=k}_{1}{\small (s)F}_{2}{\small (s),}\\
{\small F}_{2}^{\prime}{\small (s)=\varepsilon}_{3}{\small \varepsilon}%
_{4}{\small k}_{1}{\small (s)F}_{1}{\small (s)+k}_{2}{\small (s)F}%
_{3}{\small (s),}\\
{\small F}_{3}^{\prime}{\small (s)=\varepsilon}_{1}{\small \varepsilon}%
_{4}{\small k}_{2}{\small (s)F}_{2}{\small (s)+k}_{3}{\small (s)F}%
_{4}{\small (s),}\\
{\small F}_{4}^{\prime}{\small (s)=\varepsilon}_{1}{\small \varepsilon}%
_{2}{\small k}_{3}{\small (s)F}_{3}{\small (s).}%
\end{array}
\right\}  \label{frenet}%
\end{equation}

Here, $\beta(s)$ is a\ timelike curve if $\varepsilon_{1}=-1$. Also,
$\beta(s)$ is a spacelike curve with timelike principal normal $F_{2}$ or
timelike binormal $F_{3}$ or timelike trinormal $F_{4}$ if $\varepsilon
_{2}=-1$ or $\varepsilon_{3}=-1$ or $\varepsilon_{4}=-1,$ respectively, where
$g(F_{i},F_{i})=\varepsilon_{i}$, $i\in\{1,2,3,4\}.$

So, differentiating (\ref{1}) with respect to $s$ and using the Frenet formula
(\ref{frenet}), we get%
\begin{equation}
{\small \mathfrak{C}}_{s}^{\{j;\lambda\}}{\small =}\left(
{\small 1+\varepsilon}_{3}{\small \varepsilon}_{4}{\small a}_{2}{\small k}%
_{1}{\small +}a_{1_{s}}\right)  {\small F}_{1}{\small +}\left(  {\small a}%
_{1}{\small k}_{1}{\small +\varepsilon}_{1}{\small \varepsilon}_{4}%
{\small a}_{3}{\small k}_{2}{\small +}a_{2_{s}}\right)  {\small F}_{2}+\left(
{\small a}_{2}{\small k}_{2}{\small +\varepsilon}_{1}{\small \varepsilon}%
_{2}{\small a}_{4}{\small k}_{3}{\small +}a_{3_{s}}\right)  {\small F}%
_{3}{\small +}\left(  {\small a}_{3}{\small k}_{3}{\small +}a_{4_{s}}\right)
{\small F}_{4}, \label{5'}%
\end{equation}
where $\mathfrak{C}_{s}^{\{j;\lambda\}}=\frac{\partial\mathfrak{C}%
_{s}^{\{j;\lambda\}}(s,t,w)}{\partial s}$. Furthermore, $\mathfrak{C}%
^{\{j;\lambda\}}(s,t,w)-\beta(s)$ is a normal vector to the canal
hypersurfaces, which implies that%
\begin{equation}
{\small g(\mathfrak{C}}^{\{j;\lambda\}}{\small (s,t,w)-\beta(s),\mathfrak{C}%
}_{s}^{\{j;\lambda\}}{\small (s,t,w))=0} \label{6}%
\end{equation}
and so, from (\ref{1}), (\ref{5'}) and (\ref{6}) we have%
\begin{align}
&  {\small \varepsilon}_{1}{\small a}_{1}\left(  {\small 1+\varepsilon}%
_{3}{\small \varepsilon}_{4}{\small a}_{2}{\small k}_{1}{\small +}a_{1_{s}%
}\right)  +{\small \varepsilon}_{2}{\small a}_{2}\left(  {\small a}%
_{1}{\small k}_{1}{\small +\varepsilon}_{1}{\small \varepsilon}_{4}%
{\small a}_{3}{\small k}_{2}{\small +}a_{2_{s}}\right)  \text{ }\nonumber\\
&  {\small +\varepsilon}_{3}{\small a}_{3}\left(  {\small a}_{2}{\small k}%
_{2}{\small +\varepsilon}_{1}{\small \varepsilon}_{2}{\small a}_{4}%
{\small k}_{3}{\small +}a_{3_{s}}\right)  {\small +\varepsilon}_{4}%
{\small a}_{4}\left(  {\small a}_{3}{\small k}_{3}{\small +}a_{4_{s}}\right)
{\small =0.} \label{6y}%
\end{align}
Using (\ref{4}) in (\ref{6y}), we get%
\begin{equation}
{\small \varepsilon}_{1}{\small a}_{1}{\small +\varepsilon}_{1}{\small a}%
_{1}a_{1_{s}}{\small +\varepsilon}_{2}{\small a}_{2}a_{2_{s}}%
{\small +\varepsilon}_{3}{\small a}_{3}a_{3_{s}}{\small +\varepsilon}%
_{4}{\small a}_{4}a_{4_{s}}{\small =0} \label{7}%
\end{equation}
and thus, from (\ref{4}) and the definition of $\varepsilon_{i}$, we obtain
\begin{equation}
{\small a}_{1}{\small =-\varepsilon}_{1}{\small \lambda rr}_{s}{\small .}
\label{7y}%
\end{equation}
Hence, using (\ref{7y}) in (\ref{3}), we reach that%
\begin{equation}
{\small \varepsilon}_{1}{\small (\varepsilon}_{2}{\small a}_{2}^{2}%
{\small +\varepsilon}_{3}{\small a}_{3}^{2}{\small +\varepsilon}_{4}%
{\small a}_{4}^{2}{\small )=-r}^{2}{\small (-\varepsilon}_{1}{\small \lambda
+\lambda}^{2}{\small r}_{s}^{2}{\small ).} \label{8}%
\end{equation}
Here from (\ref{8});

\begin{description}
\item[i] if $\lambda=1$ or $\lambda=-1,$ then the canal hypersurfaces
$\mathfrak{C}^{\{j;\lambda\}}$ which is formed as the envelope of a family of
pseudo hyperspheres or pseudo hyperbolic hyperspheres in $E_{1}^{4}$ generated
by spacelike or timelike center curves can be parametrized by
(\ref{geneldenklemyy});

\item[ii] if $\lambda=0,$ then the canal hypersurfaces $\mathfrak{C}%
^{\{j;0\}}$ which is formed as the envelope of a family of null hypercones in
$E_{1}^{4}$ generated by spacelike or timelike center curves can be
parametrized by (\ref{geneldenklem2})
\end{description}

and this completes the proof.
\end{proof}

Thus, the explicit parametric expressions of the canal hypersurfaces
$\mathfrak{C}^{\{j;\lambda\}}(s,t,w)$ and $\mathfrak{C}^{\{j;0\}}(s,t,w)$ from
Theorem \ref{teo1} are as follows:\newline%
\begin{equation}
\left.
\begin{array}
[c]{l}%
{\small \mathfrak{C}}^{\{1;1\}}{\small =\beta+rr}^{\prime}{\small F}%
_{1}{\small \mp r}\sqrt{{\small r}^{\prime}{\small {}}^{2}+1}\left(
\cos{\small t}\cos{\small wF}_{2}{\small +}\sin{\small t}\cos{\small wF}%
_{3}{\small +}\sin{\small wF}_{4}\right)  ,\\
\\
{\small \mathfrak{C}}^{\{1;-1\}}{\small =\beta-rr}^{\prime}{\small F}%
_{1}{\small \mp r}\sqrt{{\small r}^{\prime}{\small {}}^{2}{\small -1}}\left(
\cos{\small t}\cos{\small wF}_{2}{\small +}\sin{\small t}\cos{\small wF}%
_{3}{\small +}\sin{\small wF}_{4}\right)  ,\\
\\
{\small \mathfrak{C}}^{\{2;1\}}{\small =\beta-rr}^{\prime}{\small F}%
_{1}{\small \mp r}\sqrt{{\small r}^{\prime}{\small {}}^{2}{\small -1}}\left(
\cosh{\small t}\cosh{\small wF}_{2}{\small +}\sinh{\small wF}_{3}%
{\small +}\sinh{\small t}\cosh{\small wF}_{4}\right)  ,\\
\\
{\small \mathfrak{C}}^{\{2;-1\}}{\small =\beta+rr}^{\prime}{\small F}%
_{1}{\small \mp r}\sqrt{{\small r}^{\prime}{\small {}}^{2}+1}\left(
\cosh{\small t}\cosh{\small wF}_{2}{\small +}\sinh{\small wF}_{3}%
{\small +}\sinh{\small t}\cosh{\small wF}_{4}\right)  ,\\
\\
{\small \mathfrak{C}}^{\{3;1\}}{\small =\beta-rr}^{\prime}{\small F}%
_{1}{\small \mp r}\sqrt{{\small r}^{\prime}{\small {}}^{2}{\small -1}}\left(
\sinh{\small t}\cosh{\small wF}_{2}{\small +}\cosh{\small t}\cosh
{\small wF}_{3}{\small +}\sinh{\small wF}_{4}\right)  ,\\
\\
{\small \mathfrak{C}}^{\{3;-1\}}{\small =\beta+rr}^{\prime}{\small F}%
_{1}{\small \mp r}\sqrt{{\small r}^{\prime}{\small {}}^{2}+1}\left(
\sinh{\small t}\cosh{\small wF}_{2}{\small +}\cosh{\small t}\cosh
{\small wF}_{3}{\small +}\sinh{\small wF}_{4}\right)  ,\\
\\
{\small \mathfrak{C}}^{\{4;1\}}{\small =\beta-rr}^{\prime}{\small F}%
_{1}{\small \mp r}\sqrt{{\small r}^{\prime}{\small {}}^{2}{\small -1}}\left(
\sinh{\small wF}_{2}{\small +}\sinh{\small t}\cosh{\small wF}_{3}%
{\small +}\cosh{\small t}\cosh{\small wF}_{4}\right)  ,\\
\\
{\small \mathfrak{C}}^{\{4;-1\}}{\small =\beta+rr}^{\prime}{\small F}%
_{1}{\small \mp r}\sqrt{{\small r}^{\prime}{\small {}}^{2}+1}\left(
\sinh{\small wF}_{2}{\small +}\sinh{\small t}\cosh{\small wF}_{3}%
{\small +}\cosh{\small t}\cosh{\small wF}_{4}\right)
\end{array}
\right\}  \label{canaldenk}%
\end{equation}
and%
\begin{equation}
\left.
\begin{array}
[c]{l}%
{\small \mathfrak{C}}^{\{2;0\}}{\small =\beta\mp}\sqrt{{\small a}_{3}%
^{2}{\small +a}_{4}^{2}}{\small F}_{2}{\small +a}_{3}{\small F}_{3}%
{\small +a}_{4}{\small F}_{4}{\small ,}\\
\\
{\small \mathfrak{C}}^{\{3;0\}}{\small =\beta+a}_{2}{\small F}_{2}{\small \mp
}\sqrt{{\small a}_{2}^{2}{\small +a}_{4}^{2}}{\small F}_{3}{\small +a}%
_{4}{\small F}_{4}{\small ,}\\
\\
{\small \mathfrak{C}}^{\{4;0\}}{\small =\beta+a}_{2}{\small F}_{2}%
{\small +a}_{3}{\small F}_{3}{\small \mp}\sqrt{{\small a}_{2}^{2}%
{\small +a}_{3}^{2}}{\small F}_{4}{\small .}%
\end{array}
\right\}  \label{canaldenk2}%
\end{equation}

\begin{remark}
As one can see in Theorem \ref{teo1}, we deal with the canal hypersurfaces
satisfying $r^{\prime}(s)^{2}>\lambda\varepsilon_{1}.$ If we assume that
$r^{\prime}(s)^{2}<\lambda\varepsilon_{1}$, then the canal hypersurfaces
{\small $\mathfrak{C}$}$^{\{2;1\}},$ {\small $\mathfrak{C}$}$^{\{3;1\}}$ and
{\small $\mathfrak{C}$}$^{\{4;1\}}$ can be rewritten as%
\begin{equation}
\left.
\begin{array}
[c]{l}%
{\small \mathfrak{C}}^{\{2;1\}}{\small =\beta-rr}^{\prime}{\small F}%
_{1}{\small \mp r}\sqrt{{\small 1-r}^{\prime}{\small {}}^{2}}\left(
\cosh{\small t}\sinh{\small wF}_{2}{\small +}\cosh{\small wF}_{3}%
{\small +}\sinh{\small t}\sinh{\small wF}_{4}\right)  ,\\
\\
{\small \mathfrak{C}}^{\{3;1\}}{\small =\beta-rr}^{\prime}{\small F}%
_{1}{\small \mp r}\sqrt{{\small 1-r}^{\prime}{\small {}}^{2}}\left(
\sinh{\small t\sinh wF}_{2}{\small +}\cosh{\small t\sinh wF}_{3}{\small +\cosh
wF}_{4}\right)  ,\\
\\
{\small \mathfrak{C}}^{\{4;1\}}{\small =\beta-rr}^{\prime}{\small F}%
_{1}{\small \mp r}\sqrt{{\small 1-r}^{\prime}{\small {}}^{2}}\left(
\cosh{\small wF}_{2}{\small +}\sinh{\small t\sinh wF}_{3}{\small +}%
\cosh{\small t\sinh wF}_{4}\right)  .
\end{array}
\right\}  \label{canaldenky}%
\end{equation}

\end{remark}

\subsection{\textbf{CURVATURES OF CANAL HYPERSURFACES}}

\

At the beginning of this subsection, we must note that we'll obtain the
following results by taking $r^{\prime}(s)^{2}>\lambda\varepsilon_{1}$ and
"$\mp$" which is in (\ref{geneldenklemyy}) as "$+$". Similar characterizations
can be obtained by taking $r^{\prime}(s)^{2}>\lambda\varepsilon_{1}$ and
"$\mp$" as "$-$", (or $r^{\prime}(s)^{2}<\lambda\varepsilon_{1}$and "$+$" or
"$-$").

Here, we'll obtain some important geometric invariants of these canal
hypersurfaces by proving the following theorem:

\begin{theorem}
\textit{The unit normal vector fields }$N^{\{j;\lambda\}}$\textit{, Gaussian
curvatures }$K^{\{j;\lambda\}}$\textit{, mean curvatures }$H^{\{j;\lambda\}}%
$\textit{and princaple curvatures} ${\small \mu}_{i}^{\{j;\lambda\}}%
,i\in\{1,2,3\}$, \textit{of the canal hypersurfaces }$\mathfrak{C}%
^{\{j;\lambda\}}(s,t,w)$, given by \textit{(\ref{geneldenklemyy}) in
$E_{1}^{4}$, are}%
\begin{align}
&  {\small N}^{\{j;\lambda\}}{\small =-\varepsilon}_{3}{\small \varepsilon
}_{4}{\small \lambda}^{j}\left(  {\small -\lambda\varepsilon}_{1}%
{\small r}^{\prime}{\small F}_{1}{\small (s)+}\sqrt{r^{\prime}{}^{2}%
-\lambda\varepsilon_{1}}\left(
{\textstyle\sum\nolimits_{i=2}^{4}}
\mathfrak{a}_{i}{\small (s,t,w)F}_{i}{\small (s)}\right)  \right)
{\small ,}\label{normal}\\
& \nonumber\\
&  {\small K}^{\{j;\lambda\}}{\small =}\frac{{\small \varepsilon}%
_{3}{\small \varepsilon}_{4}{\small \lambda}^{j}\left(  {\small rk}_{1}%
^{2}{\small f}_{j}^{2}\left(  r^{\prime}{}^{2}-\lambda\varepsilon_{1}\right)
{\small +r}^{\prime\prime}\left(  r^{\prime}{}^{2}-\lambda\varepsilon
_{1}+rr^{\prime\prime}\right)  {\small +\varepsilon}_{2}{\small \lambda k}%
_{1}{\small f}_{j}\sqrt{r^{\prime}{}^{2}-\lambda\varepsilon_{1}}\left(
r^{\prime}{}^{2}-\lambda\varepsilon_{1}{\small +2rr}^{\prime\prime}\right)
\right)  }{{\small r}^{2}\left(  r^{\prime}{}^{2}-\lambda\varepsilon
_{1}{\small +\varepsilon_{2}\lambda rk}_{1}{\small f}_{j}\sqrt{r^{\prime}%
{}^{2}-\lambda\varepsilon_{1}}{\small +rr}^{\prime\prime}\right)  ^{2}%
},\label{Kgenel}\\
& \nonumber\\
&  {\small H}^{\{j;\lambda\}}{\small =}\frac{{\small \varepsilon}%
_{3}{\small \varepsilon}_{4}{\small \lambda}^{^{j}}}{3}\left(  \frac{2}%
{r}{\small +}\frac{\left(  {\small rk}_{1}^{2}{\small f}_{j}^{2}\left(
r^{\prime}{}^{2}-\lambda\varepsilon_{1}\right)  {\small +r}^{\prime\prime
}\left(  r^{\prime}{}^{2}-\lambda\varepsilon_{1}{\small +rr}^{\prime\prime
}\right)  {\small +\varepsilon_{2}\lambda k}_{1}{\small f}_{j}\sqrt{r^{\prime
}{}^{2}-\lambda\varepsilon_{1}}\left(  r^{\prime}{}^{2}-\lambda\varepsilon
_{1}{\small +2rr}^{\prime\prime}\right)  \right)  }{\left(  r^{\prime}{}%
^{2}-\lambda\varepsilon_{1}+{\small \varepsilon}_{2}{\small \lambda rk}%
_{1}{\small f}_{j}\sqrt{r^{\prime}{}^{2}-\lambda\varepsilon_{1}}%
{\small +rr}^{\prime\prime}\right)  ^{2}}\right)  ,\label{Hgenel}\\
& \nonumber\\
&  \left.
\begin{array}
[c]{l}%
{\small \mu}_{1}^{\{j;\lambda\}}{\small =\mu}_{2}^{\{j;\lambda\}}%
{\small =}\frac{\varepsilon_{3}\varepsilon_{4}\lambda^{^{j}}}{r},\\
\\
{\small \mu}_{3}^{\{j;\lambda\}}{\small =}\frac{{\small \varepsilon}%
_{3}{\small \varepsilon}_{4}{\small \lambda}^{^{j}}\left(  {\small rk}_{1}%
^{2}{\small f}_{j}^{2}\left(  r^{\prime}{}^{2}-\lambda\varepsilon_{1}\right)
{\small +r}^{\prime\prime}\left(  r^{\prime}{}^{2}-\lambda\varepsilon
_{1}{\small +rr}^{\prime\prime}\right)  {\small +\varepsilon}_{2}%
{\small \lambda k}_{1}{\small f}_{j}\sqrt{r^{\prime}{}^{2}-\lambda
\varepsilon_{1}}\left(  r^{\prime}{}^{2}-\lambda\varepsilon_{1}{\small +2rr}%
^{\prime\prime}\right)  \right)  }{\left(  r^{\prime}{}^{2}-\lambda
\varepsilon_{1}{\small +\varepsilon}_{2}{\small \lambda rk}_{1}{\small f}%
_{j}\sqrt{r^{\prime}{}^{2}-\lambda\varepsilon_{1}}{\small +rr}^{\prime\prime
}\right)  ^{2}},
\end{array}
\right\}  \label{aslitoplu}%
\end{align}
where%
\begin{equation}
{\small f}_{j}{\small =}\left\{
\begin{array}
[c]{l}%
\cos{\small t}\cos{\small w;}\text{ for }{\small j=1,}\\
\cosh{\small t}\cosh{\small w;}\text{ for }{\small j=2,}\\
\sinh{\small t}\cosh{\small w;}\text{ for }{\small j=3,}\\
\sinh{\small w;}\text{ for }{\small j=4.}%
\end{array}
\right.  \label{fj}%
\end{equation}

\end{theorem}

\begin{proof}
Firstly, with the aid of the first derivatives of (\ref{canaldenk}) according
to $s,$ $t$ and $w$, we obtain the normals of the canal hypersurfaces
$\mathfrak{C}^{\{j;\lambda\}}(s,t,w)$ from (\ref{4y}) as%
\begin{equation}
\left.
\begin{array}
[c]{l}%
{\small N}^{\{1;1\}}{\small =-r}^{\prime}{\small (s)F}_{1}{\small (s)-}%
\sqrt{{\small r}^{\prime}{\small (s)}^{2}{\small +1}}\left(  \cos
{\small t}\cos{\small wF}_{2}{\small +}\sin{\small t}\cos{\small wF}%
_{3}{\small +}\sin{\small wF}_{4}\right)  {\small ,}\\
{\small N}^{\{1;-1\}}{\small =-r}^{\prime}{\small (s)F}_{1}{\small (s)+}%
\sqrt{{\small r}^{\prime}{\small (s)}^{2}{\small -1}}\left(  \cos
{\small t}\cos{\small wF}_{2}{\small +}\sin{\small t}\cos{\small wF}%
_{3}{\small +}\sin{\small wF}_{4}\right)  {\small ,}\\
{\small N}^{\{2;1\}}{\small =r}^{\prime}{\small (s)F}_{1}{\small (s)-}%
\sqrt{{\small r}^{\prime}{\small (s)}^{2}{\small -1}}\left(  \cosh
{\small t}\cosh{\small wF}_{2}{\small +}\sinh{\small wF}_{3}{\small +}%
\sinh{\small t}\cosh{\small wF}_{4}\right)  {\small ,}\\
{\small N}^{\{2;-1\}}{\small =-r}^{\prime}{\small (s)F}_{1}{\small (s)+}%
\sqrt{{\small r}^{\prime}{\small (s)}^{2}{\small +1}}\left(  \cosh
{\small t}\cosh{\small wF}_{2}{\small +}\sinh{\small wF}_{3}{\small +}%
\sinh{\small t}\cosh{\small wF}_{4}\right)  {\small ,}\\
{\small N}^{\{3;1\}}{\small =-r}^{\prime}{\small (s)F}_{1}{\small (s)+}%
\sqrt{{\small r}^{\prime}{\small (s)}^{2}{\small -1}}\left(  \sinh
{\small t}\cosh{\small wF}_{2}{\small +}\cosh{\small t}\cosh{\small wF}%
_{3}{\small +}\sinh{\small wF}_{4}\right)  {\small ,}\\
{\small N}^{\{3;-1\}}{\small =}-{\small r}^{\prime}{\small (s)F}%
_{1}{\small (s)-}\sqrt{{\small r}^{\prime}{\small (s)}^{2}{\small +1}}\left(
\sinh{\small t}\cosh{\small wF}_{2}{\small +}\cosh{\small t}\cosh
{\small wF}_{3}{\small +}\sinh{\small wF}_{4}\right)  {\small ,}\\
{\small N}^{\{4;1\}}{\small =-r}^{\prime}{\small (s)F}_{1}{\small (s)+}%
\sqrt{{\small r}^{\prime}{\small (s)}^{2}{\small -1}}\left(  \sinh
{\small wF}_{2}{\small +}\sinh{\small t}\cosh{\small wF}_{3}{\small +}%
\cosh{\small t}\cosh{\small wF}_{4}\right)  {\small ,}\\
{\small N}^{\{4;-1\}}{\small =r}^{\prime}{\small (s)F}_{1}{\small (s)+}%
\sqrt{{\small r}^{\prime}{\small (s)}^{2}{\small +1}}\left(  \sinh
{\small wF}_{2}{\small +}\sinh{\small t}\cosh{\small wF}_{3}{\small +}%
\cosh{\small t}\cosh{\small wF}_{4}\right)
\end{array}
\right\}  \label{normaller}%
\end{equation}
and so, we can write (\ref{normal}).

Now, the coefficients of the first and second fundamental forms of the canal
hypersurfaces $\mathfrak{C}^{\{1;\lambda\}}$ which are formed as the envelope
of a family of pseudo hyperspheres or pseudo hyperbolic hyperspheres in
$E_{1}^{4}$ generated by the timelike center curve are
\begin{equation}
\left.
\begin{array}
[c]{l}%
{\small g}_{11}^{\{1;\lambda\}}{\small =}\frac{{\small r}^{\prime
2}+{\small \lambda}}{4}\left(  {\small r}^{2}\left(
\begin{array}
[c]{l}%
{\small 4k}_{2}^{2}\cos^{2}{\small w-4k}_{2}{\small k}_{3}\cos{\small t}%
\sin{\small 2w}\\
{\small -}\left(  {\small 2}\cos{\small 2t}\cos^{2}{\small w+}\cos
{\small 2w-3}\right)  {\small k}_{3}^{2}%
\end{array}
\right)  {\small -4\lambda}\right) \\
\text{ \ \ \ \ \ \ \ \ }{\small -\lambda r}^{2}{\small k}_{1}^{2}\left(
\cos^{2}{\small t}\cos^{2}{\small w+}\left(  {\small \lambda}\cos
^{2}{\small t}\cos^{2}{\small w-\lambda}\right)  {\small r}^{\prime2}\right)
{\small -2\lambda rr}^{\prime\prime}{\small -}\frac{{\small \lambda}%
r^{2}r^{\prime\prime2}}{\lambda+{\small r}^{\prime2}}\\
\text{ \ \ \ \ \ \ \ \ }{\small -}\frac{2\lambda k_{1}r\cos w}{\sqrt
{\lambda+r^{\prime2}}}\left(  \left(  {\small \lambda k}_{2}{\small rr}%
^{\prime}\cos{\small t}\sin{\small t}\right)  \left(  {\small r}^{\prime
2}+{\small \lambda}\right)  {\small +\lambda rr}^{\prime\prime}\cos
{\small t}\right)  ,\\
{\small g}_{12}^{\{1;\lambda\}}{\small =g}_{21}^{\{1;\lambda\}}{\small =r}%
^{2}\left(  {\small k}_{2}\left(  {\small r}^{\prime2}+{\small \lambda
}\right)  \cos{\small w-\lambda k}_{1}{\small r}^{\prime}\sqrt{{\small r}%
^{\prime2}+{\small \lambda}}\sin{\small t-k}_{3}\left(  {\small r}^{\prime
2}+{\small \lambda}\right)  \cos{\small t}\sin{\small w}\right)
\cos{\small w},\\
{\small g}_{13}^{\{1;\lambda\}}{\small =g}_{31}^{\{1;\lambda\}}{\small =r}%
^{2}\left(  {\small k}_{3}\left(  {\small r}^{\prime2}+{\small \lambda
}\right)  \sin{\small t-\lambda k}_{1}{\small r}^{\prime}\sqrt{{\small r}%
^{\prime2}+{\small \lambda}}\cos{\small t}\sin{\small w}\right)  ,\\
{\small g}_{22}^{\{1;\lambda\}}{\small =}\left(  {\small r}^{\prime
2}+{\small \lambda}\right)  {\small r}^{2}\cos^{2}{\small w,}\text{
\ }{\small g}_{23}^{\{1;\lambda\}}{\small =g}_{32}^{\{1;\lambda\}}%
{\small =0,}\text{ \ }{\small g}_{33}^{\{1;\lambda\}}{\small =}\left(
{\small r}^{\prime2}+{\small \lambda}\right)  {\small r}^{2};
\end{array}
\right\}  \label{g1l}%
\end{equation}%
\begin{equation}
\left.
\begin{array}
[c]{l}%
{\small h}_{11}^{\{1;\lambda\}}{\small =}\frac{{\small \lambda r}\left(
{\small r}^{\prime2}+{\small \lambda}\right)  }{4}\left(  {\small 4k}_{2}%
^{2}\cos^{2}{\small w-}\left(  \cos{\small (2t)+2}\cos^{2}{\small t}%
\cos{\small (2w)-3}\right)  {\small k}_{3}^{2}{\small -4k}_{2}{\small k}%
_{3}\cos{\small t}\sin{\small (2w)}\right) \\
\text{ \ \ \ \ \ \ \ }{\small -\lambda k}_{1}^{2}{\small r}\left(
{\small \lambda}\cos^{2}{\small t}\cos^{2}{\small w+}\left(  \cos
^{2}{\small t}\cos^{2}{\small w-1}\right)  {\small r}^{\prime2}\right)
{\small -r}^{\prime\prime}{\small -}\frac{rr^{\prime\prime2}}{\lambda
+r^{\prime2}}\\
\text{ \ \ \ \ \ \ \ }{\small -}\frac{\lambda k_{1}\cos w}{\sqrt
{\lambda+r^{\prime2}}}\left(  \left(  \cos{\small t+2\lambda k}_{2}%
{\small rr}^{\prime}\sin{\small t}\right)  ({\small r}^{\prime2}%
+{\small \lambda}){\small +2rr}^{\prime\prime}\cos{\small t}\right)  ,\\
{\small h}_{12}^{\{1;\lambda\}}{\small =h}_{21}^{\{1;\lambda\}}%
{\small =\lambda r}\left(  \cos{\small wk}_{2}\left(  {\small r}^{\prime
2}+{\small \lambda}\right)  {\small -\lambda k}_{1}\sin{\small tr}^{\prime
}\sqrt{{\small r}^{\prime2}+{\small \lambda}}{\small -}\cos{\small tk}_{3}%
\sin{\small w}\left(  {\small r}^{\prime2}+{\small \lambda}\right)  \right)
\cos{\small w},\\
{\small h}_{13}^{\{1;\lambda\}}{\small =h}_{31}^{\{1;\lambda\}}%
{\small =\lambda r}\left(  {\small k}_{3}\left(  {\small r}^{\prime
2}+{\small \lambda}\right)  \sin{\small t-\lambda k}_{1}{\small r}^{\prime
}\sqrt{{\small r}^{\prime2}+{\small \lambda}}\cos{\small t}\sin{\small w}%
\right)  ,\\
{\small h}_{22}^{\{1;\lambda\}}{\small =\lambda r}\left(  {\small r}^{\prime
2}+{\small \lambda}\right)  \cos^{2}{\small w},\text{ \ }{\small h}%
_{23}^{\{1;\lambda\}}{\small =h}_{32}^{\{1;\lambda\}}{\small =0,}\text{
\ }{\small h}_{33}^{\{1;\lambda\}}{\small =\lambda r}\left(  {\small r}%
^{\prime2}+{\small \lambda}\right)  ;
\end{array}
\right\}  \text{ \ } \label{h1l}%
\end{equation}
the coefficients of the first and second fundamental forms of the canal
hypersurfaces $\mathfrak{C}^{\{2;\lambda\}}$ which are formed as the envelope
of a family of pseudo hyperspheres or pseudo hyperbolic hyperspheres in
$E_{1}^{4}$ generated by the spacelike center curve with the timelike
principal normal vector are%
\begin{equation}
\left.
\begin{array}
[c]{l}%
{\small g}_{11}^{\{2;\lambda\}}{\small =}\frac{1}{4}\left(  {\small 4-4\lambda
r}^{\prime2}{\small +r}^{2}\left(  {\small r}^{\prime2}-{\small \lambda
}\right)  \left(
\begin{array}
[c]{l}%
\left(  \cosh{\small 2t+2}\cosh^{2}{\small t}\cosh{\small 2w-3}\right)
{\small k}_{3}^{2}\\
+\left(  \cosh{\small 2t+2}\cosh{\small 2w}\sinh^{2}{\small t+3}\right)
{\small k}_{2}^{2}\\
{\small -4k}_{2}{\small k}_{3}\cosh^{2}{\small w}\sinh{\small 2t}%
\end{array}
\right)  \right) \\
\text{ \ \ \ \ \ \ \ \ }{\small +r}^{2}{\small k}_{1}^{2}\left(  \left(
\cosh^{2}{\small t}\cosh^{2}{\small w-1}\right)  {\small r}^{\prime
2}{\small -\lambda}\cosh^{2}{\small t}\cosh^{2}{\small w}\right)
{\small -2\lambda rr}^{\prime\prime}{\small -}\frac{{\small \lambda}%
r^{2}r^{\prime\prime2}}{{\small r}^{\prime2}-\lambda}\\
\text{ \ \ \ \ \ \ \ \ }{\small +}\frac{2k_{1}r}{\sqrt{{\small r}^{\prime
2}-{\small \lambda}}}\left(  \left(  \cosh{\small t}\cosh{\small w+\lambda
k}_{2}{\small rr}^{\prime}\sinh{\small w}\right)  \left(  {\small r}^{\prime
2}-{\small \lambda}\right)  {\small +rr}^{\prime\prime}\cosh{\small t}%
\cosh{\small w}\right)  ,\\
{\small g}_{12}^{\{2;\lambda\}}{\small =g}_{21}^{\{2;\lambda\}}{\small =r}%
^{2}\left(
\begin{array}
[c]{l}%
{\small \lambda}\left(  {\small k}_{1}{\small r}^{\prime}\sqrt{{\small r}%
^{\prime2}-{\small \lambda}}{\small -k}_{2}{\small \lambda}\left(
{\small r}^{\prime2}-{\small \lambda}\right)  \sinh{\small w}\right)
\sinh{\small t}\\
+{\small k}_{3}\left(  {\small r}^{\prime2}-{\small \lambda}\right)
\cosh{\small t}\sinh{\small w}%
\end{array}
\right)  \cosh{\small w},\\
{\small g}_{13}^{\{2;\lambda\}}{\small =g}_{31}^{\{2;\lambda\}}{\small =r}%
^{2}\left(  {\small \lambda k}_{1}{\small r}^{\prime}\sqrt{{\small r}%
^{\prime2}-{\small \lambda}}\cosh{\small t}\sinh{\small w+k}_{2}\left(
{\small r}^{\prime2}-{\small \lambda}\right)  \cosh{\small t+k}_{3}\left(
{\small r}^{\prime2}-{\small \lambda}\right)  \sinh{\small t}\right)  ,\\
{\small g}_{22}^{\{2;\lambda\}}{\small =}\left(  {\small r}^{\prime
2}-{\small \lambda}\right)  {\small r}^{2}\cosh^{2}{\small w,}\text{
\ }{\small g}_{23}^{\{2;\lambda\}}{\small =g}_{32}^{\{2;\lambda\}}%
{\small =0,}\text{ \ }{\small g}_{33}^{\{2;\lambda\}}{\small =r}^{2}\left(
{\small r}^{\prime2}-{\small \lambda}\right)  ;
\end{array}
\right\}  \label{g2l}%
\end{equation}%
\begin{equation}
\left.
\begin{array}
[c]{l}%
{\small h}_{11}^{\{2;\lambda\}}{\small =}\frac{1}{4}\left(
\begin{array}
[c]{l}%
{\small r}\left(  {\small r}^{\prime2}-{\small \lambda}\right)  \left(
\begin{array}
[c]{l}%
\left(  \cosh{\small (2t)+2}\cosh^{2}{\small t}\cosh{\small (2w)-3}\right)
{\small k}_{3}^{2}\\
{\small -4k}_{2}{\small k}_{3}\cosh^{2}{\small w}\sinh{\small (2t)}\\
{\small +k}_{2}^{2}\left(  {\small 3+}\cosh{\small (2t)+2}\sinh^{2}%
{\small t}\cosh{\small (2w)}\right)
\end{array}
\right)  {\small -4\lambda r}^{\prime\prime}\\
{\small +k}_{1}^{2}{\small r}\left(  {\small r}^{\prime2}\left(
\cosh{\small (2t)+2}\cosh^{2}{\small t}\cosh{\small (2w)-3}\right)
{\small -4\lambda}\cosh^{2}{\small t}\cosh^{2}{\small w}\right) \\
{\small +}\frac{4k_{1}}{\sqrt{{\small r}^{\prime2}-{\small \lambda}}}\left(
\begin{array}
[c]{l}%
\left(  \cosh{\small t}\cosh{\small w+2\lambda k}_{2}{\small rr}^{\prime}%
\sinh{\small w}\right)  \left(  {\small r}^{\prime2}-{\small \lambda}\right)
\\
{\small +2rr}^{\prime\prime}\cosh{\small t}\cosh{\small w}%
\end{array}
\right)  {\small -}\frac{4\lambda rr^{\prime\prime2}}{{\small r}^{\prime
2}-{\small \lambda}}%
\end{array}
\right)  ,\\
{\small h}_{12}^{\{2;\lambda\}}{\small =h}_{21}^{\{2;\lambda\}}{\small =r}%
\left(  {\small \lambda k}_{1}{\small r}^{\prime}\sqrt{{\small r}^{\prime
2}-{\small \lambda}}\sinh{\small t+}\left(  {\small r}^{\prime2}%
-{\small \lambda}\right)  \left(  {\small k}_{3}\cosh{\small t-k}_{2}%
\sinh{\small t}\right)  \sinh{\small w}\right)  \cosh{\small w},\\
{\small h}_{13}^{\{2;\lambda\}}{\small =h}_{31}^{\{2;\lambda\}}{\small =r}%
\left(  {\small \lambda k}_{1}{\small r}^{\prime}\sqrt{{\small r}^{\prime
2}-{\small \lambda}}\cosh{\small t}\sinh{\small w+}\left(  {\small r}%
^{\prime2}-{\small \lambda}\right)  \left(  {\small k}_{2}\cosh{\small t-k}%
_{3}\sinh{\small t}\right)  \right)  ,\\
{\small h}_{22}^{\{2;\lambda\}}{\small =r}\left(  {\small r}^{\prime
2}-{\small \lambda}\right)  \cosh^{2}{\small w},\text{ \ }{\small h}%
_{23}^{\{2;\lambda\}}{\small =h}_{32}^{\{2;\lambda\}}{\small =0,}\text{
\ }{\small h}_{33}^{\{2;\lambda\}}{\small =r}\left(  {\small r}^{\prime
2}-{\small \lambda}\right)  ;
\end{array}
\right\}  \text{ \ \ } \label{h2l}%
\end{equation}
the coefficients of the first and second fundamental forms of the canal
hypersurfaces $\mathfrak{C}^{\{3;\lambda\}}$ which are formed as the envelope
of a family of pseudo hyperspheres or pseudo hyperbolic hyperspheres in
$E_{1}^{4}$ generated by the spacelike center curve with the timelike binormal
vector are%
\begin{equation}
\left.
\begin{array}
[c]{l}%
{\small g}_{11}^{\{3;\lambda\}}{\small =}\frac{{\small r}^{\prime
2}-{\small \lambda}}{4}\left(  {\small r}^{2}\left(
\begin{array}
[c]{l}%
{\small k}_{3}^{2}\left(  {\small 3+}\cosh{\small 2t+2}\cosh{\small 2w}%
\sinh^{2}{\small t}\right) \\
+{\small 4k}_{2}^{2}\cosh^{2}{\small w-4k}_{2}{\small k}_{3}\sinh
{\small t}\sinh{\small 2w}%
\end{array}
\right)  {\small -4\lambda}\right) \\
\text{ \ \ \ \ \ \ \ \ }{\small +r}^{2}{\small k}_{1}^{2}\left(
{\small r}^{\prime2}\left(  {\small 1+}\sinh^{2}{\small t}\cosh^{2}%
{\small w}\right)  {\small -\lambda}\sinh^{2}{\small t}\cosh^{2}%
{\small w}\right)  {\small -2\lambda rr}^{\prime\prime}{\small -}%
\frac{{\small \lambda}r^{2}r^{\prime\prime2}}{{\small r}^{\prime
2}-{\small \lambda}}\\
\text{ \ \ \ \ \ \ \ \ }{\small -}\frac{{\small 2k}_{1}{\small r}%
\cosh{\small w}}{\sqrt{{\small r}^{\prime2}-{\small \lambda}}}\left(  \left(
{\small r}^{\prime2}-{\small \lambda}\right)  \left(  {\small \lambda k}%
_{2}{\small rr}^{\prime}\cosh{\small t}\sinh{\small t}\right)  +{\small rr}%
^{\prime\prime}\sinh{\small t}\right)  ,\\
{\small g}_{12}^{\{3;\lambda\}}{\small =g}_{21}^{\{3;\lambda\}}{\small =r}%
^{2}\left(
\begin{array}
[c]{l}%
{\small k}_{2}\left(  {\small r}^{\prime2}-{\small \lambda}\right)
\cosh{\small w-\lambda k}_{1}{\small r}^{\prime}\sqrt{{\small r}^{\prime
2}-{\small \lambda}}\cosh{\small t}\\
{\small -k}_{3}\left(  {\small r}^{\prime2}{\small -\lambda}\right)
\sinh{\small t}\sinh{\small w}%
\end{array}
\right)  \cosh^{2}{\small w},\\
{\small g}_{13}^{\{3;\lambda\}}{\small =g}_{31}^{\{3;\lambda\}}{\small =r}%
^{2}\left(  {\small k}_{3}\left(  {\small r}^{\prime2}-{\small \lambda
}\right)  \cosh{\small t-\lambda k}_{1}{\small r}^{\prime}\sqrt{{\small r}%
^{\prime2}-{\small \lambda}}\sinh{\small t}\sinh{\small w}\right)  ,\\
{\small g}_{22}^{\{3;\lambda\}}{\small =}\left(  {\small r}^{\prime
2}-{\small \lambda}\right)  {\small r}^{2}\cosh^{2}{\small w,}\text{
\ }{\small g}_{23}^{\{3;\lambda\}}{\small =g}_{32}^{\{3;\lambda\}}%
{\small =0,}\text{ \ }{\small g}_{33}^{\{3;\lambda\}}{\small =}\left(
{\small r}^{\prime2}-{\small \lambda}\right)  {\small r}^{2};
\end{array}
\right\}  \text{ \ \ \ \ \ \ \ \ \ \ } \label{g3l}%
\end{equation}%
\begin{equation}
\left.
\begin{array}
[c]{l}%
{\small h}_{11}^{\{3;\lambda\}}{\small =2k}_{1}{\small k}_{2}{\small rr}%
^{\prime}\sqrt{{\small r}^{\prime2}-{\small \lambda}}\cosh{\small t}%
\cosh{\small w-\lambda k}_{2}^{2}{\small r}\left(  {\small r}^{\prime
2}-{\small \lambda}\right)  \cosh^{2}{\small w}\\
\text{ \ \ \ \ \ \ \ \ }{\small -}\frac{\lambda k_{3}^{2}r\left(
{\small r}^{\prime2}-{\small \lambda}\right)  }{4}\left(  {\small 3+}%
\cosh{\small (2t)+2}\sinh^{2}{\small t}\cosh{\small (2w)}\right)
{\small +\lambda k}_{2}{\small k}_{3}{\small r}\left(  {\small r}^{\prime
2}-{\small \lambda}\right)  \sinh{\small t}\sinh{\small (2w)}\\
\text{ \ \ \ \ \ \ \ \ }{\small +}\frac{k_{1}^{2}r}{4}\left(  {\small 4}%
\cosh^{2}{\small w}\sinh^{2}{\small t-\lambda r}^{\prime2}\left(
{\small 3+}\cosh{\small (2t)+2}\sinh^{2}{\small t}\cosh{\small (2w)}\right)
\right)  {\small +r}^{\prime\prime}{\small +}\frac{rr^{\prime\prime2}%
}{-\lambda+r^{\prime2}}\\
\text{ \ \ \ \ \ \ \ \ }{\small +}\frac{\lambda}{\sqrt{{\small r}^{\prime
2}-{\small \lambda}}}\left(  {\small k}_{1}\left(  {\small r}^{\prime
2}-{\small \lambda}{\small +2rr}^{\prime\prime}\right)  \cosh{\small w}%
\sinh{\small t}\right)  ,\\
{\small h}_{12}^{\{3;\lambda\}}{\small =h}_{21}^{\{3;\lambda\}}{\small =r}%
\left(  {\small k}_{1}{\small r}^{\prime}\sqrt{{\small r}^{\prime
2}-{\small \lambda}}\cosh{\small t-\lambda}\left(  {\small r}^{\prime
2}-{\small \lambda}\right)  \left(  {\small k}_{2}\cosh{\small w}%
+{\small k}_{3}\sinh{\small t}\sinh{\small w}\right)  \right)  \cosh
{\small w},\\
{\small h}_{13}^{\{3;\lambda\}}{\small =h}_{31}^{\{3;\lambda\}}{\small =r}%
\left(  {\small k}_{1}{\small r}^{\prime}\sqrt{{\small r}^{\prime
2}-{\small \lambda}}\sinh{\small t}\sinh{\small w-\lambda k}_{3}\left(
{\small r}^{\prime2}-{\small \lambda}\right)  \cosh{\small t}\right)  ,\\
{\small h}_{22}^{\{3;\lambda\}}{\small =-\lambda r}\left(  {\small r}%
^{\prime2}-{\small \lambda}\right)  \cosh^{2}{\small w},\text{ \ }%
{\small h}_{23}^{\{3;\lambda\}}{\small =h}_{32}^{\{3;\lambda\}}{\small =0,}%
\text{ \ }{\small h}_{33}^{\{3;\lambda\}}{\small =-\lambda r}\left(
{\small r}^{\prime2}-{\small \lambda}\right)
\end{array}
\right\}  \label{h3l}%
\end{equation}
and the coefficients of the first and second fundamental forms of the canal
hypersurfaces $\mathfrak{C}^{\{4;\lambda\}}$ which are formed as the envelope
of a family of pseudo hyperspheres or pseudo hyperbolic hyperspheres in
$E_{1}^{4}$ generated by the spacelike center curve with the timelike
trinormal vector are%
\begin{equation}
\left.
\begin{array}
[c]{l}%
{\small g}_{11}^{\{4;\lambda\}}{\small =}\frac{-\lambda+{\small r}^{\prime2}%
}{4}\left(  {\small r}^{2}\left(
\begin{array}
[c]{l}%
{\small k}_{2}^{2}\left(  {\small 2}\cosh{\small 2t}\cosh^{2}{\small w+}%
\cosh{\small 2w-3}\right) \\
{\small +4k}_{3}^{2}\cosh^{2}{\small w+4k}_{2}{\small k}_{3}\cosh
{\small t}\sinh{\small 2w}%
\end{array}
\right)  {\small -4\lambda}\right) \\
\text{ \ \ \ \ \ \ \ }{\small +r}^{2}{\small k}_{1}^{2}\left(  {\small r}%
^{\prime2}\cosh^{2}{\small w-\lambda}\sinh^{2}{\small w}\right)
{\small -2\lambda rr}^{\prime\prime}{\small -}\frac{{\small \lambda}%
r^{2}r^{\prime\prime2}}{-\lambda+{\small r}^{\prime2}}\\
\text{ \ \ \ \ \ \ \ }{\small +}\frac{2\lambda k_{1}r}{\sqrt{{\small r}%
^{\prime2}-{\small \lambda}}}\left(  {\small k}_{2}{\small rr}^{\prime}\left(
{\small r}^{\prime2}-{\small \lambda}\right)  \cosh{\small w}\sinh
{\small t-\lambda}\left(  {\small r}^{\prime2}-{\small \lambda}{\small +rr}%
^{\prime\prime}\right)  \sinh{\small w}\right)  ,\\
{\small g}_{12}^{\{4;\lambda\}}{\small =g}_{21}^{\{4;\lambda\}}{\small =r}%
^{2}\left(  {\small r}^{\prime2}-{\small \lambda}\right)  \left(
{\small k}_{2}\cosh{\small t}\sinh{\small w+k}_{3}\cosh{\small w}\right)
\cosh{\small w},\\
{\small g}_{13}^{\{4;\lambda\}}{\small =g}_{31}^{\{4;\lambda\}}{\small =-r}%
^{2}\left(  {\small \lambda k}_{1}{\small r}^{\prime}\sqrt{{\small r}%
^{\prime2}-{\small \lambda}}\cosh{\small w+k}_{2}\left(  {\small r}^{\prime
2}-{\small \lambda}\right)  \sinh{\small t}\right)  ,\\
{\small g}_{22}^{\{4;\lambda\}}{\small =r}^{2}\left(  {\small r}^{\prime
2}-{\small \lambda}\right)  \cosh^{2}{\small w},\text{ \ }{\small g}%
_{23}^{\{4;\lambda\}}{\small =g}_{32}^{\{4;\lambda\}}{\small =0,}\text{
\ }{\small g}_{33}^{\{4;\lambda\}}{\small =r}^{2}\left(  {\small r}^{\prime
2}-{\small \lambda}\right)
\end{array}
\right\}  \text{
\ \ \ \ \ \ \ \ \ \ \ \ \ \ \ \ \ \ \ \ \ \ \ \ \ \ \ \ \ \ \ \ } \label{g4l}%
\end{equation}%
\begin{equation}
\left.
\begin{array}
[c]{l}%
{\small h}_{11}^{\{4;\lambda\}}{\small =}\frac{-{\small r}\left(
{\small r}^{\prime2}-{\small \lambda}\right)  }{4}\left(  {\small k}_{2}%
^{2}\left(  \cosh{\small (2t)+2}\cosh^{2}{\small t}\cosh{\small (2w)-3}%
\right)  {\small +4k}_{3}(k_{3}\cosh^{2}{\small w+k}_{2}\cosh{\small t}%
\sinh{\small (2w))}\right) \\
\text{ \ \ \ \ \ \ \ \ }{\small +\lambda k}_{1}^{2}{\small r}\left(  \sinh
^{2}{\small w-\lambda r}^{\prime2}\cosh^{2}{\small w}\right)  {\small +\lambda
r}^{\prime\prime}{\small +}\frac{\lambda rr^{\prime\prime2}}{{\small r}%
^{\prime2}-{\small \lambda}}\\
\text{ \ \ \ \ \ \ \ \ }{\small +}\frac{k_{1}}{\sqrt{{\small r}^{\prime
2}-{\small \lambda}}}\left(  \left(  \sinh{\small w-2\lambda k}_{2}%
{\small rr}^{\prime}\sinh{\small t\cosh w}\right)  \left(  {\small r}^{\prime
2}-{\small \lambda}\right)  {\small +2rr}^{\prime\prime}\sinh{\small w}%
\right)  ,\\
{\small h}_{12}^{\{4;\lambda\}}{\small =h}_{21}^{\{4;\lambda\}}{\small =-r}%
\left(  {\small k}_{3}\cosh{\small w+k}_{2}\cosh{\small t}\sinh{\small w}%
\right)  \left(  {\small r}^{\prime2}-{\small \lambda}\right)  \cosh
{\small w},\\
{\small h}_{13}^{\{4;\lambda\}}{\small =h}_{31}^{\{4;\lambda\}}{\small =r}%
\left(  {\small \lambda k}_{1}{\small r}^{\prime}\sqrt{{\small r}^{\prime
2}-{\small \lambda}}\cosh{\small w+k}_{2}\left(  {\small r}^{\prime
2}-{\small \lambda}\right)  \sinh{\small t}\right)  ,\\
{\small h}_{22}^{\{4;\lambda\}}{\small =-r}\left(  {\small r}^{\prime
2}-{\small \lambda}\right)  \cosh^{2}{\small w},\text{ \ }{\small h}%
_{23}^{\{4;\lambda\}}{\small =h}_{32}^{\{4;\lambda\}}{\small =0,}\text{
\ }{\small h}_{33}^{\{4;\lambda\}}{\small =-r}\left(  {\small r}^{\prime
2}-{\small \lambda}\right)
\end{array}
\right\}  \label{h4l}%
\end{equation}
and these imply that
\begin{equation}
\det{\small g}_{ij}^{\{j;\lambda\}}{\small =-\lambda A}^{2}{\small r}%
^{4}\left(  r^{\prime}{}^{2}-\lambda\varepsilon_{1}\right)  \left(  r^{\prime
}{}^{2}-\lambda\varepsilon_{1}{\small +\varepsilon}_{2}{\small \lambda k}%
_{1}{\small f}_{j}{\small r}\sqrt{r^{\prime}{}^{2}-\lambda\varepsilon_{1}%
}{\small +rr}^{\prime\prime}\right)  ^{2},\text{
\ \ \ \ \ \ \ \ \ \ \ \ \ \ \ \ \ \ \ } \label{detgtoplu}%
\end{equation}%
\begin{equation}
\det{\small h}_{ij}^{\{j;\lambda\}}{\small =-\lambda A}^{2}{\small r}%
^{2}\left(  r^{\prime}{}^{2}-\lambda\varepsilon_{1}\right)  \left(
\begin{array}
[c]{l}%
{\small \varepsilon}_{3}{\small \varepsilon}_{4}{\small \lambda}^{j}%
{\small f}_{j}^{2}{\small k}_{1}^{2}{\small r}\left(  r^{\prime}{}^{2}%
-\lambda\varepsilon_{1}\right)  {\small +\varepsilon}_{3}{\small \varepsilon
}_{4}{\small \lambda}^{j}{\small r}^{\prime\prime}\left(  r^{\prime}{}%
^{2}-\lambda\varepsilon_{1}{\small +rr}^{\prime\prime}\right) \\
{\small +\varepsilon}_{2}{\small \varepsilon}_{3}{\small \varepsilon}%
_{4}{\small \lambda}^{j+1}{\small f}_{j}{\small k}_{1}\sqrt{r^{\prime}{}%
^{2}-\lambda\varepsilon_{1}}\left(  r^{\prime}{}^{2}-\lambda\varepsilon
_{1}{\small +2rr}^{\prime\prime}\right)
\end{array}
\right)  , \label{dethtoplu}%
\end{equation}
where $A=\cos w,$ for $j=1$ and $A=\cosh w,$ for $j=2,3,4.$ Thus, from
(\ref{yy4}), (\ref{detgtoplu}) and (\ref{dethtoplu}), the Gaussian curvatures
are obtained as (\ref{Kgenel}).

Here, obtaining the inverse of the first fundamental forms and using these and
second fundamental forms in (\ref{7yyy}), we obtain the components of the
shape operators of canal hypersurfaces $\mathfrak{C}^{\{j;\lambda\}}$ as%
\begin{equation}
\left.
\begin{array}
[c]{l}%
{\small S}_{11}^{\{1;\lambda\}}{\small =}\frac{{\small \lambda k}_{1}%
^{2}{\small r}\left(  {\small r}^{\prime2}+{\small \lambda}\right)  \cos
^{2}{\small t}\cos^{2}{\small w+\lambda r}^{\prime\prime}\left(  r^{\prime
2}+\lambda+rr^{\prime\prime}\right)  {\small k}_{1}\sqrt{{\small r}^{\prime
2}+{\small \lambda}}\left(  {\small r}^{\prime2}+{\small \lambda}%
{\small +2rr}^{\prime\prime}\right)  \cos{\small t}\cos{\small w}}{\left(
{\small r}^{\prime2}+{\small \lambda}{\small +r}\left(  {\small \lambda k}%
_{1}\sqrt{{\small r}^{\prime2}+{\small \lambda}}\cos t\cos w{\small +r}%
^{\prime\prime}\right)  \right)  ^{2}},\\
{\small S}_{21}^{\{1;\lambda\}}{\small =}\frac{{\small -\lambda}%
\sqrt{{\small r}^{\prime2}+{\small \lambda}}\left(  {\small k}_{1}%
{\small r}^{\prime}\sin{\small t}\sec{\small w-\lambda k}_{2}\sqrt
{{\small r}^{\prime2}+{\small \lambda}}{\small +\lambda k}_{3}\sqrt
{{\small r}^{\prime2}+{\small \lambda}}\cos{\small t}\tan{\small w}\right)
\sec{\small w}}{{\small r}\left(  {\small k}_{1}{\small r}\sqrt{{\small r}%
^{\prime2}+{\small \lambda}}\cos{\small t+\lambda}\left(  {\small r}^{\prime
2}+{\small \lambda}{\small +rr}^{\prime\prime}\right)  \sec{\small w}\right)
},\\
{\small S}_{31}^{\{1;\lambda\}}{\small =}\frac{{\small \lambda}\left(
{\small \lambda k}_{3}\left(  {\small r}^{\prime2}+{\small \lambda}\right)
\sin{\small t}-{\small k}_{1}{\small r}^{\prime}\sqrt{{\small r}^{\prime
2}+{\small \lambda}}\cos{\small t}\sin{\small w}\right)  }{{\small r}\left(
{\small 1+\lambda r}^{\prime2}{\small +r}\left(  {\small k}_{1}\sqrt
{\lambda+r^{\prime2}}\cos{\small t}\cos{\small w+\lambda r}^{\prime\prime
}\right)  \right)  },\\
{\small S}_{22}^{\{1;\lambda\}}{\small =S}_{33}^{\{1;\lambda\}}{\small =}%
\frac{\lambda}{r},\\
{\small S}_{12}^{\{1;\lambda\}}{\small =S}_{13}^{\{1;\lambda\}}{\small =S}%
_{23}^{\{1;\lambda\}}{\small =S}_{32}^{\{1;\lambda\}}{\small =0;}%
\end{array}
\right\}  \text{ \ \ \ \ \ \ \ \ \ } \label{S1l}%
\end{equation}

\begin{equation}
\left.
\begin{array}
[c]{l}%
{\small S}_{11}^{\{2;\lambda\}}{\small =}\frac{{\small k}_{1}^{2}%
{\small r}\left(  {\small r}^{\prime2}-{\small \lambda}\right)  \cosh
^{2}{\small t}\cosh^{2}{\small w+r}^{\prime\prime}\left(  {\small r}^{\prime
2}-{\small \lambda}{\small +rr}^{\prime\prime}\right)  {\small -\lambda k}%
_{1}\sqrt{{\small r}^{\prime2}-{\small \lambda}}\left(  {\small r}^{\prime
2}-{\small \lambda}{\small +2rr}^{\prime\prime}\right)  \cosh{\small t}%
\cosh{\small w}}{\left(  {\small r}^{\prime2}-{\small \lambda}{\small +r}%
\left(  {\small -\lambda k}_{1}\sqrt{{\small r}^{\prime2}-{\small \lambda}%
}\cosh{\small t}\cosh{\small w+r}^{\prime\prime}\right)  \right)  ^{2}},\\
{\small S}_{21}^{\{2;\lambda\}}{\small =}\frac{\left(  {\small \lambda k}%
_{1}{\small r}^{\prime}\sqrt{{\small r}^{\prime2}-{\small \lambda}}%
\sinh{\small t}\sec h{\small w+}\left(  {\small k}_{3}\cosh{\small t-k}%
_{2}\sinh{\small t}\right)  \sqrt{{\small r}^{\prime2}-{\small \lambda}}%
\tan{\small w}\right)  \sec{\small hw}}{{\small r}\left(  \left(
{\small r}^{\prime2}-{\small \lambda}{\small +rr}^{\prime\prime}\right)
\sec{\small hw-\lambda k}_{1}{\small r}\sqrt{{\small r}^{\prime2}%
-{\small \lambda}}\cosh{\small t}\right)  },\\
{\small S}_{31}^{\{2;\lambda\}}{\small =}\frac{{\small -k}_{1}{\small r}%
^{\prime}\sqrt{{\small r}^{\prime2}-{\small \lambda}}\cosh{\small t}%
\sinh{\small w-\lambda k}_{2}\left(  {\small r}^{\prime2}-{\small \lambda
}\right)  \cosh{\small t+\lambda k}_{3}\left(  {\small r}^{\prime
2}-{\small \lambda}\right)  \sinh{\small t}}{{\small r}\left(
{\small 1-\lambda r}^{\prime2}{\small +r}\left(  {\small k}_{1}\sqrt
{{\small r}^{\prime2}-{\small \lambda}}\cosh{\small t}\cosh{\small w-\lambda
r}^{\prime\prime}\right)  \right)  },\\
{\small S}_{22}^{\{2;\lambda\}}{\small =S}_{33}^{\{2;\lambda\}}{\small =}%
\frac{1}{r},\\
{\small S}_{12}^{\{2;\lambda\}}{\small =S}_{13}^{\{2;\lambda\}}{\small =S}%
_{23}^{\{2;\lambda\}}{\small =S}_{32}^{\{2;\lambda\}}{\small =0;}%
\end{array}
\right\}  \text{ \ \ \ \ } \label{S2l}%
\end{equation}

\begin{equation}
\left.
\begin{array}
[c]{l}%
{\small S}_{11}^{\{3;\lambda\}}{\small =-\lambda}\frac{{\small k}_{1}%
^{2}{\small r}\left(  {\small r}^{\prime2}-{\small \lambda}\right)  \sinh
^{2}{\small t}\cosh^{2}{\small w+r}^{\prime\prime}\left(  {\small r}^{\prime
2}-{\small \lambda}{\small +rr}^{\prime\prime}\right)  {\small +\lambda k}%
_{1}\sqrt{{\small r}^{\prime2}-{\small \lambda}}\left(  {\small r}^{\prime
2}-{\small \lambda}{\small +2rr}^{\prime\prime}\right)  \sinh{\small t}%
\cosh{\small w}}{\left(  {\small r}^{\prime2}-{\small \lambda}{\small +r}%
\left(  {\small \lambda k}_{1}\sqrt{{\small r}^{\prime2}-{\small \lambda}%
}\sinh{\small t}\cosh{\small w+r}^{\prime\prime}\right)  \right)  ^{2}%
}{\small ,}\\
{\small S}_{21}^{\{3;\lambda\}}{\small =}\frac{\sec{\small hw}\left(
{\small k}_{1}{\small r}^{\prime}\sqrt{{\small r}^{\prime2}-{\small \lambda}%
}\cosh{\small t}\sec{\small hw-\lambda}\left(  {\small k}_{2}{\small -k}%
_{3}\sinh{\small t}\tanh{\small w}\right)  \left(  {\small r}^{\prime
2}-{\small \lambda}\right)  \right)  }{r\left(  \lambda k_{1}r\sqrt
{{\small r}^{\prime2}-{\small \lambda}}\sinh t+\left(  -\lambda+r^{\prime
2}+rr^{\prime\prime}\right)  \sec hw\right)  }{\small ,}\\
{\small S}_{31}^{\{3;\lambda\}}{\small =}\frac{-k_{1}r^{\prime}\sqrt
{{\small r}^{\prime2}-{\small \lambda}}\sinh t\sinh w-\lambda k_{3}\left(
{\small r}^{\prime2}-{\small \lambda}\right)  \cosh t}{r\left(  r^{\prime
2}-\lambda+r\left(  \lambda k_{1}\sqrt{-\lambda+r^{\prime2}}\sinh t\cosh
w+r^{\prime\prime}\right)  \right)  }{\small ,}\\
{\small S}_{22}^{\{3;\lambda\}}{\small =S}_{33}^{\{3;\lambda\}}{\small =-}%
\frac{\lambda}{r},\\
{\small S}_{12}^{\{3;\lambda\}}{\small =S}_{13}^{\{3;\lambda\}}{\small =S}%
_{23}^{\{3;\lambda\}}{\small =S}_{32}^{\{3;\lambda\}}{\small =0;}%
\end{array}
\right\}  \label{S3l}%
\end{equation}

\begin{equation}
\left.
\begin{array}
[c]{l}%
{\small S}_{11}^{\{4;\lambda\}}{\small =}\frac{-k_{1}^{2}r\left(
{\small r}^{\prime2}-{\small \lambda}\right)  \sinh^{2}w-r^{\prime\prime
}\left(  r^{\prime2}-\lambda+rr^{\prime\prime}\right)  -\lambda k_{1}%
\sqrt{{\small r}^{\prime2}-{\small \lambda}}\left(  r^{\prime2}-\lambda
+2rr^{\prime\prime}\right)  \sinh w}{\left(  r^{\prime2}-\lambda+r\left(
\lambda k_{1}\sqrt{-\lambda+r^{\prime2}}\sinh w+r^{\prime\prime}\right)
\right)  ^{2}},\\
{\small S}_{21}^{\{4;\lambda\}}{\small =}-{\small \lambda}\frac{\left(
k_{3}+k_{2}\tanh w\cosh t\right)  \left(  {\small r}^{\prime2}-{\small \lambda
}\right)  }{r\left(  \lambda r^{\prime2}-1+r\left(  k_{1}\sqrt{{\small r}%
^{\prime2}-{\small \lambda}}\sinh w+\lambda r^{\prime\prime}\right)  \right)
},\\
{\small S}_{31}^{\{4;\lambda\}}{\small =}\frac{-k_{1}r^{\prime}\sqrt
{{\small r}^{\prime2}-{\small \lambda}}\cosh w+\lambda k_{2}\left(
{\small r}^{\prime2}-{\small \lambda}\right)  \sinh t}{r\left(  \lambda
r^{\prime2}-1+r\left(  k_{1}\sqrt{{\small r}^{\prime2}-{\small \lambda}}\sinh
w+\lambda r^{\prime\prime}\right)  \right)  },\\
{\small S}_{22}^{\{4;\lambda\}}{\small =S}_{33}^{\{4;\lambda\}}{\small =-}%
\frac{1}{r},\\
{\small S}_{12}^{\{4;\lambda\}}{\small =S}_{13}^{\{4;\lambda\}}{\small =S}%
_{23}^{\{4;\lambda\}}{\small =S}_{32}^{\{4;\lambda\}}{\small =0.}%
\end{array}
\right\}  \text{ \ \ \ \ \ \ \ \ \ \ \ \ \ \ } \label{S4l}%
\end{equation}
From (\ref{yy5}) and (\ref{S1l})-(\ref{S4l}), we obtain the mean curvatures of
the canal hypersurfaces $\mathfrak{C}^{\{j;\lambda\}}$ as (\ref{Hgenel}).

Finally, from (\ref{S1l})-(\ref{S4l}), we get%
\begin{equation}
\det(S-\mu I)=\left(  \mu-\frac{\varepsilon_{3}\varepsilon_{4}\lambda^{^{j}}%
}{r}\right)  ^{2}\left(  \frac{\left(
\begin{array}
[c]{l}%
{\small \varepsilon}_{3}{\small \varepsilon}_{4}{\small \lambda}^{^{j}}\left(
{\small rk}_{1}^{2}{\small f}_{j}^{2}\left(  r^{\prime}{}^{2}-\lambda
\varepsilon_{1}\right)  {\small +r}^{\prime\prime}\left(  r^{\prime}{}%
^{2}-\lambda\varepsilon_{1}+rr^{\prime\prime}\right)  \right) \\
{\small +\varepsilon}_{2}{\small \varepsilon}_{3}{\small \varepsilon}%
_{4}{\small \lambda}^{^{j+1}}{\small k}_{1}{\small f}_{j}\sqrt{r^{\prime}%
{}^{2}-\lambda\varepsilon_{1}}\left(  r^{\prime}{}^{2}-\lambda\varepsilon
_{1}{\small +2rr}^{\prime\prime}\right)
\end{array}
\right)  }{\left(  r^{\prime}{}^{2}-\lambda\varepsilon_{1}{\small +\varepsilon
}_{2}{\small \lambda rk}_{1}{\small f}_{j}\sqrt{r^{\prime}{}^{2}%
-\lambda\varepsilon_{1}}{\small +rr}^{\prime\prime}\right)  ^{2}}-\mu\right)
\label{det(S-I)}%
\end{equation}
and solving the equation of $\det(S-\mu I)=0$ we have (\ref{aslitoplu}). So,
the proof is completed.
\end{proof}

Here, from (\ref{normaller}) or (\ref{detgtoplu}), we can state the following proposition:

\begin{proposition}
The canal hypersurfaces $\mathfrak{C}^{\{j;\lambda\}}(s,t,w)$ in $E_{1}^{4}$
are spacelike or timelike if $\lambda=-1$ or $\lambda=1$, respectively.
\end{proposition}

\subsection{\textbf{SOME GEOMETRIC CHARACTERIZATIONS FOR CANAL HYPERSURFACES}}

\

From (\ref{Kgenel}) and (\ref{Hgenel}), we can obtain the following important
relation between the Gaussian and mean curvatures of the canal hypersurfaces
$\mathfrak{C}^{\{j;\lambda\}}$:

\begin{theorem}
\textit{The Gaussian and mean curvatures of the canal hypersurfaces
}$\mathfrak{C}^{\{j;\lambda\}}(s,t,w),$ given by \textit{(\ref{geneldenklemyy}%
) in }$E_{1}^{4},$ \textit{satisfy}%
\begin{equation}
3H^{\{j;\lambda\}}r-K^{\{j;\lambda\}}r^{3}-2\varepsilon_{3}\varepsilon
_{4}\lambda^{j}=0. \label{yhk}%
\end{equation}

\end{theorem}

\begin{theorem}
The canal hypersurfaces $\mathfrak{C}^{\{j;\lambda\}}(s,t,w),$ given by
\textit{(\ref{geneldenklemyy})} in $E_{1}^{4},$ are flat if and only if
$k_{1}=0$ and $r(s)=as+b$, $a,b\in%
\mathbb{R}
,$ $a\neq\mp1$ .
\end{theorem}

\begin{proof}
Firstly, let we suppose that the canal hypersurfaces $\mathfrak{C}%
^{\{j;\lambda\}}(s,t,w)$ in $E_{1}^{4}$ are flat. Then from (\ref{Kgenel}) we
get
\begin{equation}
\text{ }{\small rk}_{1}^{2}\left(  r^{\prime}{}^{2}-\lambda\varepsilon
_{1}\right)  {\small f}_{j}^{2}{\small +\varepsilon}_{2}{\small \lambda}%
^{j}{\small k}_{1}\sqrt{r^{\prime}{}^{2}-\lambda\varepsilon_{1}}\left(
r^{\prime}{}^{2}-\lambda\varepsilon_{1}{\small +2rr}^{\prime\prime}\right)
{\small f}_{j}{\small +r}^{\prime\prime}\left(  r^{\prime}{}^{2}%
-\lambda\varepsilon_{1}+rr^{\prime\prime}\right)  =0. \label{flat1}%
\end{equation}

Since $\{f_{j},f_{j}^{2},1\}$ are linearly independent, we have%
\begin{equation}
{\small rk}_{1}^{2}\left(  r^{\prime}{}^{2}-\lambda\varepsilon_{1}\right)
={\small \varepsilon}_{2}{\small \lambda}^{j}{\small k}_{1}\sqrt{r^{\prime}%
{}^{2}-\lambda\varepsilon_{1}}\left(  r^{\prime}{}^{2}-\lambda\varepsilon
_{1}{\small +2rr}^{\prime\prime}\right)  ={\small r}^{\prime\prime}\left(
r^{\prime}{}^{2}-\lambda\varepsilon_{1}+rr^{\prime\prime}\right)  =0.
\label{flat2}%
\end{equation}

From the (\ref{flat2}), we get $k_{1}=0$ and ${\small r}^{\prime\prime}\left(
r^{\prime}{}^{2}-\lambda\varepsilon_{1}+rr^{\prime\prime}\right)  =0.$ When
$k_{1}=0$, from (\ref{Kgenel}) we have $r^{\prime}{}^{2}-\lambda
\varepsilon_{1}+rr^{\prime\prime}\neq0$ and so it must be $r^{\prime\prime}=0$.

Conversely, if $k_{1}=0$ and $r(s)=as+b$, $a,b\in%
\mathbb{R}
,$ $a\neq\mp1$, from (\ref{Kgenel}) we get $K=0$ and this completes the proof.
\end{proof}

\begin{theorem}
The canal hypersurfaces $\mathfrak{C}^{\{j;\lambda\}}(s,t,w),$ given by
\textit{(\ref{geneldenklemyy})} in $E_{1}^{4},$ are minimal if and only if
$k_{1}=0$ and the radius $r(s)$ satisfies $%
{\displaystyle\int}
\frac{dr}{\sqrt{{\small \varepsilon_{1}\lambda}+\left(  \frac{c_{1}}%
{r}\right)  ^{\frac{4}{3}}}}=\pm s+c_{2},$ $c_{1},c_{2}\in%
\mathbb{R}
$.
\end{theorem}

\begin{proof}
Firstly, let we suppose that the canal hypersurfaces $\mathfrak{C}%
^{\{j;\lambda\}}(s,t,w)$ in $E_{1}^{4}$ are minimal. Then from (\ref{Hgenel})
we get%
\begin{align}
&  -3{\small r}^{2}{\small k}_{1}^{2}\left(  {\small \varepsilon}%
_{1}{\small \lambda-r}^{\prime}{\small {}}^{2}\right)  f_{j}^{2}%
-{\small \varepsilon_{2}\lambda}k_{1}r\sqrt{r^{\prime}{}^{2}-\lambda
\varepsilon_{1}}\left(  5{\small \varepsilon_{1}\lambda-5r}^{\prime
2}-6rr^{\prime\prime}\right)  {\small f}_{j}\nonumber\\
&  \text{ }+2+2r^{\prime4}-5{\small \varepsilon_{1}\lambda rr}^{\prime\prime
}+3r^{2}r^{\prime\prime2}+r^{\prime2}\left(  -4{\small \varepsilon_{1}%
\lambda+5}rr^{\prime\prime}\right)  =0. \label{min1}%
\end{align}
Since $\{f_{j},f_{j}^{2},1\}$ are linearly independent, we have%
\begin{equation}
\left.
\begin{array}
[c]{l}%
-3{\small r}^{2}{\small k}_{1}^{2}\left(  {\small \varepsilon}_{1}%
{\small \lambda-r}^{\prime}{\small {}}^{2}\right)  =0,\\
{\small \varepsilon_{2}\lambda}k_{1}r\sqrt{r^{\prime}{}^{2}-\lambda
\varepsilon_{1}}\left(  5{\small \varepsilon_{1}\lambda-5r}^{\prime
2}-6rr^{\prime\prime}\right)  =0,\\
2+2r^{\prime4}-5{\small \varepsilon_{1}\lambda rr}^{\prime\prime}%
+3r^{2}r^{\prime\prime2}+r^{\prime2}\left(  -4{\small \varepsilon_{1}%
\lambda+5}rr^{\prime\prime}\right)  =0.
\end{array}
\right\}  \label{min2}%
\end{equation}
From (\ref{min2}), we get $k_{1}=0$ and $2+2r^{\prime4}-5{\small \varepsilon
_{1}\lambda rr}^{\prime\prime}+3r^{2}r^{\prime\prime2}+r^{\prime2}\left(
-4{\small \varepsilon_{1}\lambda+5}rr^{\prime\prime}\right)  =0;$ i.e.%
\begin{equation}
\left(  2{\small \varepsilon_{1}\lambda}-2r^{\prime2}-3{\small rr}%
^{\prime\prime}\right)  \left(  {\small \varepsilon_{1}\lambda-}r^{\prime
2}-{\small rr}^{\prime\prime}\right)  =0. \label{min3}%
\end{equation}
When $k_{1}=0$, from (\ref{Hgenel}) we have $r^{\prime}{}^{2}-\lambda
\varepsilon_{1}+rr^{\prime\prime}\neq0$ and so it must be%
\begin{equation}
-2(r^{\prime2}-{\small \varepsilon_{1}\lambda})-3{\small rr}^{\prime\prime
}=0\text{.} \label{min4}%
\end{equation}

Now, let us solve the equation (\ref{min4}).

If we take $r^{\prime}{}(s)=h(s),$ we get
\begin{equation}
r^{\prime\prime}=h^{\prime}=\frac{dh}{dr}\frac{dr}{ds}=\frac{dh}{dr}h.
\label{min5}%
\end{equation}
Using (\ref{min5}) in (\ref{min4}), we have%
\begin{equation}
3r\frac{dh}{dr}h+2h^{2}-2{\small \varepsilon_{1}\lambda}=0. \label{min6}%
\end{equation}
From (\ref{min4}), $r^{\prime}{}(s)=h(s)\neq0$ and so we reach that%
\begin{equation}
\frac{3h}{2({\small \varepsilon_{1}\lambda}-h^{2})}dh=\frac{dr}{r}.
\label{min7}%
\end{equation}
By integrating (\ref{min7}), we have%
\begin{equation}
\text{ }h=\pm\sqrt{{\small \varepsilon_{1}\lambda}+\left(  \frac{c_{1}}%
{r}\right)  ^{\frac{4}{3}}}, \label{min8}%
\end{equation}
where $c_{1}$ is constant. Since $r^{\prime}{}=\frac{dr}{ds}=h,$ from
(\ref{min8}) we get%
\begin{equation}%
{\displaystyle\int}
\frac{dr}{\sqrt{{\small \varepsilon_{1}\lambda}+\left(  \frac{c_{1}}%
{r}\right)  ^{\frac{4}{3}}}}=\pm%
{\displaystyle\int}
ds. \label{min9y}%
\end{equation}

Conversely, if $k_{1}=0$ and $r(s)$ satisfies $%
{\displaystyle\int}
\frac{dr}{\sqrt{{\small \varepsilon_{1}\lambda}+\left(  \frac{c_{1}}%
{r}\right)  ^{\frac{4}{3}}}}=\pm s+c_{2},$ $c_{1},c_{2}\in%
\mathbb{R}
$, then we have $H=0$ and this completes the proof.
\end{proof}

Now, if%
\begin{equation}
H_{s}K_{t}-H_{t}K_{s}=0,\text{ }H_{s}K_{w}-H_{w}K_{s}=0,\text{ }H_{t}%
K_{w}-H_{w}K_{t}=0, \label{w1}%
\end{equation}
hold on a hypersurface, then we call the hypersurface as $(H,K)_{st}%
$-Weingarten, $(H,K)_{sw}$-Weingarten, $(H,K)_{tw}$-Weingarten hypersurface,
respectively, where $H_{s}=\frac{\partial H}{\partial s}$ and so on. So, from
(\ref{Kgenel}) and (\ref{Hgenel}) we have

\begin{theorem}
The canal hypersurfaces $\mathfrak{C}^{\{j;\lambda\}}(s,t,w),$ given by
\textit{(\ref{geneldenklemyy})} in $E_{1}^{4},$ are $(H,K)_{tw}$-Weingarten.
\end{theorem}

\begin{proof}
From (\ref{Kgenel}) and (\ref{Hgenel}), we get $H_{t}^{\{j;\lambda\}}%
K_{w}^{\{j;\lambda\}}-H_{w}^{\{j;\lambda\}}K_{t}^{\{j;\lambda\}}=0$ and this
completes the proof.
\end{proof}

\begin{theorem}
\label{swteo}The canal hypersurfaces $\mathfrak{C}^{\{j;\lambda\}}(s,t,w),$
given by \textit{(\ref{geneldenklemyy}) in }$E_{1}^{4}$, are $(H,K)_{sw}%
$-Weingarten if and only if $k_{1}=0$ or $r(s)=$constant$.$
\end{theorem}

\begin{proof}
From (\ref{Kgenel}) and (\ref{Hgenel}), we have%
\begin{align}
&  H_{s}^{\{j;\lambda\}}K_{w}^{\{j;\lambda\}}-H_{w}^{\{j;\lambda\}}%
K_{s}^{\{j;\lambda\}}\label{sw}\\
&  =-\frac{\left(
\begin{array}
[c]{l}%
2{\small \varepsilon_{3}^{2}\varepsilon_{4}^{2}}\lambda^{2j}k_{1}r^{\prime
}\left(  r^{\prime}{}^{2}-\lambda\varepsilon_{1}\right)  ^{2}\\
\left(  \left(  {\small \varepsilon_{2}^{2}\lambda}^{2}-1\right)  f_{j}%
^{2}k_{1}^{2}r^{2}+r^{\prime2}-{\small \lambda\varepsilon}_{1}%
+{\small \varepsilon}_{2}{\small \lambda}k_{1}r\sqrt{r^{\prime}{}^{2}%
-\lambda\varepsilon_{1}}{\small f}_{j}+rr^{\prime\prime}\right) \\
\left(
\begin{array}
[c]{l}%
{\small \varepsilon}_{2}{\small \lambda}\sqrt{r^{\prime}{}^{2}-\lambda
\varepsilon_{1}}\left(  {\small r}^{\prime2}-{\small \lambda\varepsilon}%
_{1}+rr^{\prime\prime}\right) \\
+k_{1}r\left(  \left(  2-{\small \varepsilon}_{2}^{2}{\small \lambda}%
^{2}\right)  \left(  r^{\prime2}-{\small \lambda\varepsilon}_{1}\right)
-2\left(  {\small \varepsilon}_{2}^{2}{\small \lambda}^{2}-1\right)
rr^{\prime\prime}\right)  f_{j}%
\end{array}
\right)  f_{j_{w}}%
\end{array}
\right)  }{3r^{4}\left(  {\small r}^{\prime2}-{\small \lambda\varepsilon}%
_{1}+{\small \varepsilon}_{2}{\small \lambda}k_{1}r\sqrt{r^{\prime}{}%
^{2}-\lambda\varepsilon_{1}}{\small f}_{j}+rr^{\prime\prime}\right)  ^{5}%
}.\nonumber
\end{align}
If $H_{s}^{\{j;\lambda\}}K_{w}^{\{j;\lambda\}}-H_{w}^{\{j;\lambda\}}%
K_{s}^{\{j;\lambda\}}=0$, then from (\ref{sw}) we have%
\begin{equation}
2k_{1}r^{\prime}\left(  r^{\prime}{}^{2}-\lambda\varepsilon_{1}\right)
^{2}\left(
\begin{array}
[c]{l}%
r^{\prime}{}^{2}-\lambda\varepsilon_{1}+rr^{\prime\prime}\\
+{\small \varepsilon}_{2}{\small \lambda}k_{1}r\sqrt{r^{\prime}{}^{2}%
-\lambda\varepsilon_{1}}{\small f}_{j}%
\end{array}
\right)  \left(
\begin{array}
[c]{l}%
{\small \varepsilon}_{2}{\small \lambda}\sqrt{r^{\prime}{}^{2}-\lambda
\varepsilon_{1}}\left(  r^{\prime}{}^{2}-\lambda\varepsilon_{1}+rr^{\prime
\prime}\right) \\
+k_{1}r\left(  r^{\prime}{}^{2}-\lambda\varepsilon_{1}\right)  f_{j}%
\end{array}
\right)  f_{j_{w}}=0 \label{sw2}%
\end{equation}
and so,%
\begin{equation}
k_{1}r^{\prime}\left(
\begin{array}
[c]{l}%
r^{\prime}{}^{2}-\lambda\varepsilon_{1}+rr^{\prime\prime}\\
+{\small \varepsilon}_{2}{\small \lambda}k_{1}r\sqrt{r^{\prime}{}^{2}%
-\lambda\varepsilon_{1}}{\small f}_{j}%
\end{array}
\right)  \left(
\begin{array}
[c]{l}%
{\small \varepsilon}_{2}{\small \lambda}\sqrt{r^{\prime}{}^{2}-\lambda
\varepsilon_{1}}\left(  r^{\prime}{}^{2}-\lambda\varepsilon_{1}+rr^{\prime
\prime}\right) \\
+k_{1}r\left(  r^{\prime}{}^{2}-\lambda\varepsilon_{1}\right)  f_{j}%
\end{array}
\right)  =0. \label{sw3}%
\end{equation}
From (\ref{sw}), the second and third component of (\ref{sw3}) cannot be zero
and so it must be
\begin{equation}
k_{1}r^{\prime}=0. \label{sw4}%
\end{equation}
This completes the proof.
\end{proof}

\begin{theorem}
The canal hypersurfaces $\mathfrak{C}^{\{j;\lambda\}}(s,t,w),$ given by
\textit{(\ref{geneldenklemyy}) in }$E_{1}^{4}$, are $(H,K)_{st}$-Weingarten
and the canal hypersurfaces $\mathfrak{C}^{\{1;\lambda\}}(s,t,w),$
$\mathfrak{C}^{\{2;\lambda\}}(s,t,w),$ $\mathfrak{C}^{\{3;\lambda\}}(s,t,w)$
are $(H,K)_{st}$-Weingarten if and only if $k_{1}=0$ or $r(s)=$constant$.$
\end{theorem}

\begin{proof}
From (\ref{Kgenel}) and (\ref{Hgenel}), we have%
\begin{align}
&  H_{s}^{\{j;\lambda\}}K_{t}^{\{j;\lambda\}}-H_{t}^{\{j;\lambda\}}%
K_{s}^{\{j;\lambda\}}\label{st}\\
&  =-\frac{\left(
\begin{array}
[c]{l}%
2{\small \varepsilon_{3}^{2}\varepsilon_{4}^{2}}\lambda^{2j}k_{1}r^{\prime
}\left(  r^{\prime}{}^{2}-\lambda\varepsilon_{1}\right)  ^{2}\\
\left(  \left(  {\small \varepsilon_{2}^{2}\lambda}^{2}-1\right)  k_{1}%
^{2}r^{2}f_{j}^{2}+r^{\prime2}-{\small \lambda\varepsilon}_{1}%
+{\small \varepsilon}_{2}{\small \lambda}k_{1}r\sqrt{r^{\prime}{}^{2}%
-\lambda\varepsilon_{1}}{\small f}_{j}+rr^{\prime\prime}\right) \\
\left(
\begin{array}
[c]{l}%
{\small \varepsilon}_{2}{\small \lambda}\sqrt{r^{\prime}{}^{2}-\lambda
\varepsilon_{1}}\left(  r^{\prime}{}^{2}-\lambda\varepsilon_{1}+rr^{\prime
\prime}\right) \\
+k_{1}r\left(  \left(  2-{\small \varepsilon}_{2}^{2}{\small \lambda}%
^{2}\right)  \left(  r^{\prime2}-{\small \lambda\varepsilon}_{1}\right)
-2\left(  {\small \varepsilon}_{2}^{2}{\small \lambda}^{2}-1\right)
rr^{\prime\prime}\right)  f_{j}%
\end{array}
\right)  f_{j_{t}}%
\end{array}
\right)  }{3r^{4}\left(  r^{\prime}{}^{2}-\lambda\varepsilon_{1}%
+{\small \varepsilon}_{2}{\small \lambda}k_{1}r\sqrt{r^{\prime}{}^{2}%
-\lambda\varepsilon_{1}}{\small f}_{j}+rr^{\prime\prime}\right)  ^{5}%
}.\nonumber
\end{align}
Using (\ref{fj}), for $j=1,2,3$ the proof can be seen with similar method in
Theorem \ref{swteo}. Also, for $j=4$ we have $f_{4_{t}}=0$ and so
$H_{s}^{\{4;\lambda\}}K_{t}^{\{4;\lambda\}}-H_{t}^{\{4;\lambda\}}%
K_{s}^{\{4;\lambda\}}=0$ and this completes the proof.
\end{proof}

\section{\textbf{TUBULAR HYPERSURFACES\ GENERATED BY NON-NULL CURVES IN
}$E_{1}^{4}$}

In this section, we give the above results, which are given for the canal
hypersurfaces in $E_{1}^{4},$ for tubular hypersurfaces.

\begin{theorem}
\textit{The tubular hypersurfaces }$\mathfrak{T}^{\{j;\lambda\}}%
{\small (s,t,w)}$\textit{ which are formed as the envelope of a family of
pseudo hyperspheres or pseudo hyperbolic hyperspheres in }$E_{1}^{4}$\textit{
generated by spacelike or timelike center curves} \textit{can be parametrized
by}
\begin{equation}
\left.
\begin{array}
[c]{l}%
\mathfrak{T}^{\{1;1\}}{\small (s,t,w)=\beta(s)\mp r(s)}\left(  {\small \cos
t\cos wF}_{2}{\small (s)+\sin t\cos wF}_{3}+\sin wF_{4}\right)  ,\\
\mathfrak{T}^{\{2;1\}}{\small (s,t,w)=\beta(s)\mp r(s)}\left(  {\small \cosh
t\sinh wF}_{2}{\small (s)+\cosh wF}_{3}+\sinh t\sinh wF_{4}\right)  ,\\
\mathfrak{T}^{\{2;-1\}}{\small (s,t,w)=\beta(s)\mp r(s)}\left(  {\small \cosh
t\cosh wF}_{2}{\small (s)+\sinh wF}_{3}+\sinh t\cosh wF_{4}\right)  ,\\
\mathfrak{T}^{\{3;1\}}{\small (s,t,w)=\beta(s)\mp r(s)}\left(  {\small \sinh
t\sinh wF}_{2}{\small (s)+\cosh t\sinh wF}_{3}+\cosh wF_{4}\right)  ,\\
\mathfrak{T}^{\{3;-1\}}{\small (s,t,w)=\beta(s)\mp r(s)}\left(  {\small \sinh
t\cosh wF}_{2}{\small (s)+\cosh t\cosh wF}_{3}+\sinh wF_{4}\right)  ,\\
\mathfrak{T}^{\{4;1\}}{\small (s,t,w)=\beta(s)\mp r(s)}\left(  {\small \cosh
wF}_{2}{\small (s)+\sinh t\sinh wF}_{3}+\cosh t\sinh wF_{4}\right)  ,\\
\mathfrak{T}^{\{4;-1\}}{\small (s,t,w)=\beta(s)\mp r(s)}\left(  {\small \sinh
wF}_{2}{\small (s)+\sinh t\cosh wF}_{3}+\cosh t\cosh wF_{4}\right)  .
\end{array}
\right\}  \label{Tub}%
\end{equation}
\textit{Also, }there is no \textit{tubular hypersurface }$\mathfrak{T}%
^{\{1;-1\}}{\small (s,t,w)}$\textit{ which is formed as the envelope of a
family of pseudo hyperbolic hyperspheres in }$E_{1}^{4}$\textit{ generated by
timelike center curves}.

Furthermore t\textit{he tubular hypersurfaces }$\mathfrak{T}^{\{j;0\}}%
{\small (s,t,w)}$\textit{ which are formed as the envelope of a family of null
hypercones in }$E_{1}^{4}$\textit{ generated by spacelike or timelike center
curves} \textit{can be parametrized by} (\ref{canaldenk2}).
\end{theorem}

\begin{proof}
Let the center curve $\beta:I\subseteq%
\mathbb{R}
\rightarrow E_{1}^{4}$ be a arc-length parametrized timelike or spacelike
curve with non-zero curvature. Then, the parametrization of the envelope of
pseudo hyperspheres (resp. pseudo hyperbolic hyperspheres or null hypercones)
defining the tubular hypersurfaces $\mathfrak{T}^{\{j;\lambda\}}$ in
$E_{1}^{4}$ can be given by%
\begin{equation}
\mathfrak{T}^{\{j;\lambda\}}{\small (s,t,w)-\beta(s)=a_{1}(s,t,w)F}%
_{1}{\small (s)+a}_{2}{\small (s,t,w)F}_{2}{\small (s)+a_{3}(s,t,w)F}%
_{3}{\small (s)+a_{4}(s,t,w)F}_{4}{\small (s).} \label{1T}%
\end{equation}
Furthermore, since $\mathfrak{T}^{\{j;\lambda\}}(s,t,w)$ lies on the pseudo
hyperspheres (resp. pseudo hyperbolic hyperspheres or null hypercones), we
have%
\begin{equation}
{\small g(}\mathfrak{T}^{\{j;\lambda\}}{\small (s,t,w)-\beta(s),}%
\mathfrak{T}^{\{j;\lambda\}}{\small (s,t,w)-\beta(s))=\lambda r}^{2}
\label{2T}%
\end{equation}
which leads to from (\ref{1T}) that%
\begin{equation}
{\small \varepsilon}_{1}{\small a}_{1}^{2}{\small +\varepsilon}_{2}%
{\small a}_{2}^{2}{\small +\varepsilon}_{3}{\small a}_{3}^{2}%
{\small +\varepsilon}_{4}{\small a}_{4}^{2}{\small =\lambda r}^{2} \label{3T}%
\end{equation}
and%
\begin{equation}
{\small \varepsilon}_{1}{\small a}_{1}{\small a}_{1_{s}}{\small +\varepsilon
}_{2}{\small a}_{2}{\small a}_{2_{s}}{\small +\varepsilon}_{3}{\small a}%
_{3}{\small a}_{3_{s}}{\small +\varepsilon}_{4}{\small a}_{4}{\small a}%
_{4_{s}}{\small =0,} \label{4T}%
\end{equation}
where $\lambda=1$, $\lambda=-1$ and $\lambda=0$ if the tubular hypersurfaces
are obtained by pseudo hyperspheres, pseudo hyperbolic hyperspheres and null
hypercones, respectively; $r$ is constant radius $a_{i}=a_{i}(s,t,w),$
$a_{i_{s}}=\frac{\partial a_{i}(s,t,w)}{\partial s},$ and so on.

With similar procedure in the proof of the Theorem \ref{teo1}, we reach that
\begin{equation}
{\small a=0.} \label{7yT}%
\end{equation}
Hence, using (\ref{7yT}) in (\ref{3T}), we have%
\begin{equation}
{\small \varepsilon}_{2}{\small a}_{2}^{2}{\small +\varepsilon}_{3}%
{\small a}_{3}^{2}{\small +\varepsilon}_{4}{\small a}_{4}^{2}{\small =\lambda
r}^{2}{\small .} \label{8T}%
\end{equation}
From (\ref{1T}) and (\ref{8T}), we obtain the parametric expressions of the
tubular hypersurfaces $\mathfrak{T}^{\{j;\lambda\}}(s,t,w)$ as (\ref{Tub}) and
(\ref{canaldenk2}).
\end{proof}

Now, let we give the Gaussian and mean curvatures of tubular hypersurfaces
$\mathfrak{T}^{\{j;\lambda\}}(s,t,w)$ given by (\ref{Tub}) with similar method
used for canal hypersurfaces.

\begin{theorem}
The Gaussian and mean curvatures of tubular hypersurfaces $\mathfrak{T}%
^{\{j;\lambda\}}(s,t,w)$, given by (\ref{Tub}) in $E_{1}^{4}$, are
\begin{equation}
\left.
\begin{array}
[c]{l}%
{\small K}^{\{1;1\}}=\frac{k_{1}\cos t{\small \cos w}}{r^{2}\left(
1+rk_{1}\cos t{\small \cos w}\right)  },~{\small H}^{\{1;1\}}=\frac
{2+3rk_{1}\cos t{\small \cos w}}{3r\left(  1+rk_{1}\cos t{\small \cos
w}\right)  };\\
\\
{\small K}^{\{2;1\}}=\frac{k_{1}{\small \cosh t\sinh w}}{r^{2}\left(
1+rk_{1}{\small \cosh t\sinh w}\right)  },~{\small H}^{\{2;1\}}=\frac
{2+3rk_{1}{\small \cosh t\sinh w}}{3r\left(  1+rk_{1}{\small \cosh t\sinh
w}\right)  };\\
\\
{\small K}^{\{2;-1\}}=\frac{k_{1}{\small \cosh t\cosh w}}{r^{2}\left(
1+rk_{1}{\small \cosh t\cosh w}\right)  },~{\small H}^{\{2;-1\}}%
=\frac{2+3rk_{1}{\small \cosh t\cosh w}}{3r\left(  1+rk_{1}{\small \cosh
t\cosh w}\right)  };\\
\\
{\small K}^{\{3;1\}}=\frac{k_{1}{\small \sinh t\sinh w}}{r^{2}\left(
1-rk_{1}{\small \sinh t\sinh w}\right)  },~{\small H}^{\{3;1\}}=\frac
{2-3rk_{1}{\small \sinh t\sinh w}}{3r\left(  -1+rk_{1}{\small \sinh t\sinh
w}\right)  };\\
\\
{\small K}^{\{3;-1\}}=\frac{k_{1}{\small \sinh t\cosh w}}{r^{2}\left(
-1+rk_{1}{\small \sinh t\cosh w}\right)  },~{\small H}^{\{3;-1\}}%
=\frac{2-3rk_{1}{\small \sinh t\cosh w}}{3r\left(  1-rk_{1}{\small \sinh
t\cosh w}\right)  };\\
\\
{\small K}^{\{4;1\}}=\frac{k_{1}{\small \cosh w}}{r^{2}\left(  1-rk_{1}%
{\small \cosh w}\right)  },~{\small H}^{\{4;1\}}=\frac{2-3rk_{1}{\small \cosh
w}}{3r\left(  -1+rk_{1}{\small \cosh w}\right)  };\\
\\
{\small K}^{\{4;-1\}}=\frac{k_{1}{\small \sinh w}}{r^{2}\left(  1-rk_{1}%
{\small \sinh w}\right)  },~{\small H}^{\{4;-1\}}=\frac{2-3rk_{1}{\small \sinh
w}}{3r\left(  -1+rk_{1}{\small \sinh w}\right)  }.
\end{array}
\right\}  \label{Tcurvature}%
\end{equation}

\end{theorem}

From (\ref{Tcurvature}), we can state the following theorem:

\begin{theorem}
The tubular hypersurfaces $\mathfrak{T}^{\{j;\lambda\}}(s,t,w)$, given by
(\ref{Tub}) in $E_{1}^{4}$, are $(H,K)_{tw}$-Weingarten, $(H,K)_{ts}%
$-Weingarten and $(H,K)_{sw}$-Weingarten hypersurfaces.
\end{theorem}

\section{\textbf{VISUALIZATION}}

In this section, we construct examples for canal hypersurfaces $\mathfrak{C}%
^{\{1;\lambda\}}(s,t,w)$ and $\mathfrak{C}^{\{3;\lambda\}}(s,t,w)$ which are
formed as the envelope of a family of pseudo hyperspheres or pseudo hyperbolic
hyperspheres in\textit{ }$E_{1}^{4}$ with the aid of a timelike curve and a
spacelike curve with timelike binormal vector, seperately.

Firstly, let we take the unit speed timelike curve%
\begin{equation}
{\small \beta}_{1}{\small (s)=}\left(  2\sinh s,2\cosh s,\sqrt{3}\cos
s,\sqrt{3}\sin s\right)  \label{ex1y}%
\end{equation}
in $E_{1}^{4}$. The Frenet vectors and curvatures of the curve (\ref{ex1y})
are%
\begin{equation}
\left.
\begin{array}
[c]{l}%
F_{1}^{{\small \beta}_{1}}=\left(  2\cosh s,2\sinh s,-\sqrt{3}\sin s,\sqrt
{3}\cos s\right)  ,\\
F_{2}^{{\small \beta}_{1}}=\left(  \frac{2}{\sqrt{7}}\sinh s,\frac{2}{\sqrt
{7}}\cosh s,-\sqrt{\frac{3}{7}}\cos s,-\sqrt{\frac{3}{7}}\sin s\right)  ,\\
F_{3}^{{\small \beta}_{1}}=\left(  -\sqrt{3}\cosh s,-\sqrt{3}\sinh s,2\sin
s,-2\cos s\right)  ,\\
F_{4}^{{\small \beta}_{1}}=\left(  \sqrt{\frac{3}{7}}\sinh s,\sqrt{\frac{3}%
{7}}\cosh s,\frac{2}{\sqrt{7}}\cos s,\frac{2}{\sqrt{7}}\sin s\right)  ,\\
\\
k_{1}^{{\small \beta}_{1}}=\sqrt{7},\text{ }k_{2}^{{\small \beta}_{1}}%
=4\sqrt{\frac{3}{7}},\text{ }k_{3}^{{\small \beta}_{1}}=\frac{1}{\sqrt{7}}.
\end{array}
\right\}  \label{ex2y}%
\end{equation}
From (\ref{canaldenk}), the canal hypersurfaces $\mathfrak{C}^{\{1;\lambda
\}}(s,t,w)$ which are formed as the envelope of a family of pseudo
hyperspheres or pseudo hyperbolic hyperspheres in\textit{ }$E_{1}^{4}$
generated by the curve (\ref{ex1y}) are%
\begin{align}
&  \mathfrak{C}^{\{1;1\}}(s,t,w)=\label{ex3y}\\
&  \left(
\begin{array}
[c]{c}%
-2s\cosh s\left(  -4+\sqrt{15}\cos w\sin t\right)  +\frac{2}{7}\sinh s\left(
7+\sqrt{35}s\left(  2\cos t\cos w+\sqrt{3}\sin w\right)  \right)  ,\\
-2s\sinh s\left(  -4+\sqrt{15}\cos w\sin t\right)  +\frac{2}{7}\cosh s\left(
7+\sqrt{35}s\left(  2\cos t\cos w+\sqrt{3}\sin w\right)  \right)  ,\\
-4s\sin s\left(  \sqrt{3}-\sqrt{5}\cos w\sin t\right)  +\cos s\left(  \sqrt
{3}-2\sqrt{\frac{5}{7}}s\left(  \sqrt{3}\cos t\cos w-2\sin w\right)  \right)
,\\
\sin s\left(  \sqrt{3}-2\sqrt{\frac{5}{7}}s\left(  \sqrt{3}\cos t\cos w-2\sin
w\right)  \right)  +4s\cos s\left(  \sqrt{3}-\sqrt{5}\cos w\sin t\right)
\nonumber
\end{array}
\right)
\end{align}
and%
\begin{align}
&  \mathfrak{C}^{\{1;-1\}}(s,t,w)=\label{ex4y}\\
&  \left(
\begin{array}
[c]{l}%
-2s\left(  4+3\cos w\sin t\right)  \cosh s+\frac{2}{7}\left(  7+2\sqrt
{21}s\cos t\cos w+3\sqrt{7}s\sin w\right)  \sinh s,\\
-2s\left(  4+3\cos w\sin t\right)  \sinh s+\frac{2}{7}\left(  7+\sqrt
{21}s\left(  2\cos t\cos w+\sqrt{3}\sin w\right)  \right)  \cosh s,\\
4\sqrt{3}s\left(  1+\cos w\sin t\right)  \sin s+\left(  \sqrt{3}-\frac
{2}{\sqrt{7}}s\left(  3\cos t\cos w-2\sqrt{3}\sin w\right)  \right)  \cos s,\\
\sin s\left(  \sqrt{3}-\frac{2}{\sqrt{7}}s\left(  3\cos t\cos w-2\sqrt{3}\sin
w\right)  \right)  -4\sqrt{3}s\left(  1+\cos w\sin t\right)  \cos s
\end{array}
\right)  ,\text{ \ \ \ \ \ }\nonumber
\end{align}
where the radius function has been taken as $r(s)=2s.$ From (\ref{Kgenel}),
(\ref{Hgenel}) and (\ref{aslitoplu}), the Gaussian, mean and principal
curvatures of the canal hypersurfaces $\mathfrak{C}^{\{1;1\}}$ and
$\mathfrak{C}^{\{1;-1\}}$ are obtained as%
\begin{equation}
\left.
\begin{array}
[c]{l}%
K^{\{1;1\}}=\frac{5\left(  \sqrt{35}+14s\cos t\cos w\right)  \cos t\cos
w}{4s^{2}\left(  5+2\sqrt{35}s\cos t\cos w\right)  ^{2}},\\
H^{\{1;1\}}=\frac{1}{3}\left(  \frac{1}{s}+\frac{5\left(  \sqrt{35}+14s\cos
t\cos w\right)  \cos t\cos w}{\left(  5+2\sqrt{35}s\cos t\cos w\right)  ^{2}%
}\right)  ,\\
{\small \mu}_{1}^{\{1;1\}}{\small =\mu}_{2}^{\{1;1\}}{\small =}\frac{1}%
{2s},\text{ }{\small \mu}_{3}^{\{1;1\}}{\small =}\frac{5\left(  \sqrt
{35}+14s\cos t\cos w\right)  \cos t\cos w}{\left(  5+2\sqrt{35}s\cos t\cos
w\right)  ^{2}}%
\end{array}
\right\}  \label{ex5y}%
\end{equation}
and%
\begin{equation}
\left.
\begin{array}
[c]{l}%
K^{\{1;-1\}}=\frac{3\left(  \sqrt{21}-14s\cos t\cos w\right)  \cos t\cos
w}{4s^{2}\left(  3-2\sqrt{21}s\cos t\cos w\right)  ^{2}},\\
H^{\{1;-1\}}=\frac{-3+5\sqrt{21}s\cos t\cos w-42s^{2}\cos^{2}t\cos^{2}%
w}{s\left(  3-2\sqrt{21}s\cos t\cos w\right)  ^{2}},\\
{\small \mu}_{1}^{\{1;-1\}}{\small =\mu}_{2}^{\{1;-1\}}{\small =}\frac{-1}%
{2s},\text{ }{\small \mu}_{3}^{\{1;-1\}}{\small =}\frac{3\left(  \sqrt
{21}-14s\cos t\cos w\right)  \cos t\cos w}{\left(  3-2\sqrt{21}s\cos t\cos
w\right)  ^{2}},
\end{array}
\right\}  \label{ex6y}%
\end{equation}
respectively. In the following figures, one can see the projections of the
canal hypersurfaces (\ref{ex3y}) and (\ref{ex4y}) for $w=2$ and $r(s)=2s$ into
$x_{2}x_{3}x_{4}$-spaces in (A) and (B), respectively.

\begin{figure}[H]
\centering
\includegraphics[
height=2.8in, width=4.9in
]{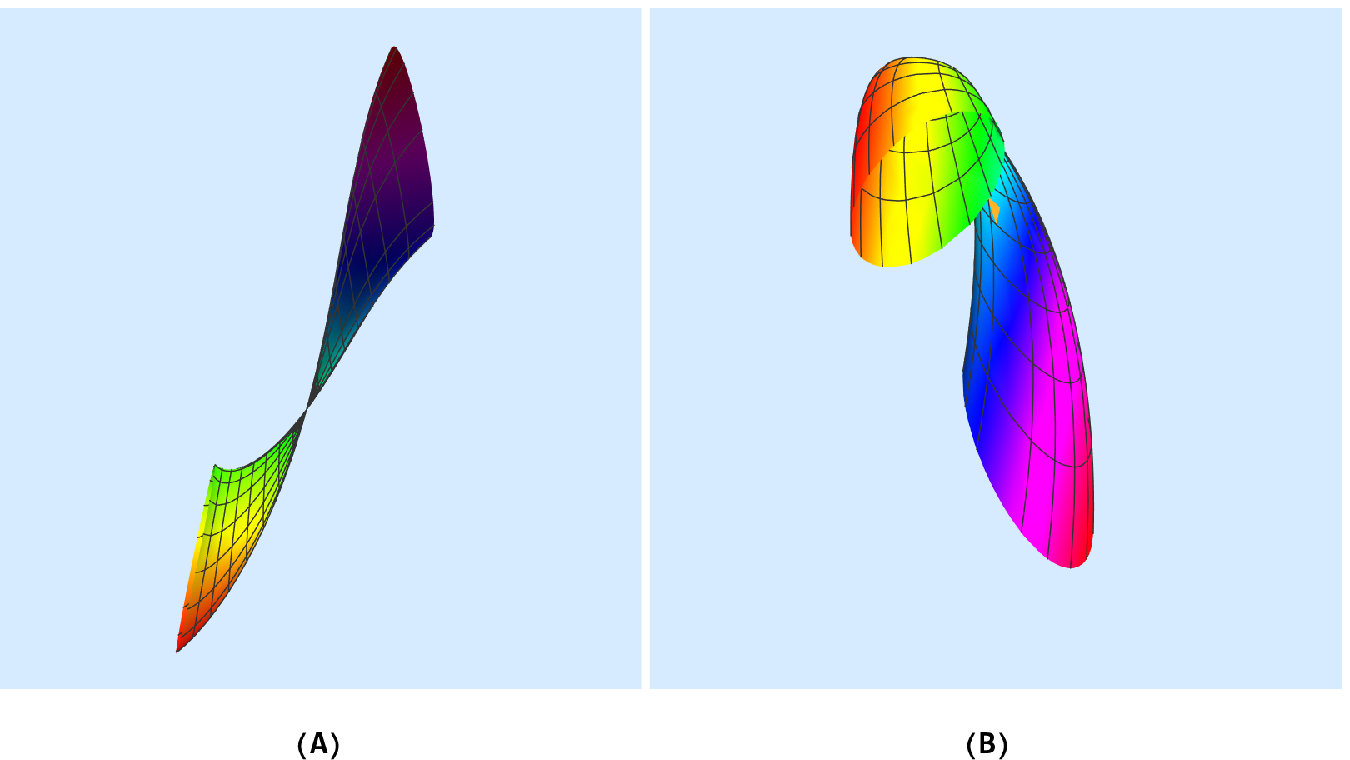}\caption{ }%
\label{fig:1}%
\end{figure}

Secondly, let we take the unit speed spacelike curve with timelike binormal
vector%
\begin{equation}
{\small \beta}_{2}{\small (s)=}\left(  \sqrt{3}\sinh s,\sqrt{3}\cosh s,2\cos
s,2\sin s\right)  \label{ex1}%
\end{equation}
in $E_{1}^{4}$. The Frenet vectors and curvatures of the curve (\ref{ex1}) are%
\begin{equation}
\left.
\begin{array}
[c]{l}%
F_{1}^{{\small \beta}_{2}}=\left(  \sqrt{3}\cosh s,\sqrt{3}\sinh s,-2\sin
s,2\cos s\right)  ,\\
F_{2}^{{\small \beta}_{2}}=\left(  \sqrt{\frac{3}{7}}\sinh s,\sqrt{\frac{3}%
{7}}\cosh s,-\frac{2}{\sqrt{7}}\cos s,-\frac{2}{\sqrt{7}}\sin s\right)  ,\\
F_{3}^{{\small \beta}_{2}}=\left(  2\cosh s,2\sinh s,-\sqrt{3}\sin s,\sqrt
{3}\cos s\right)  ,\\
F_{4}^{{\small \beta}_{2}}=\left(  \frac{2}{\sqrt{7}}\sinh s,\frac{2}{\sqrt
{7}}\cosh s,\sqrt{\frac{3}{7}}\cos s,\sqrt{\frac{3}{7}}\sin s\right)  ,\\
\\
k_{1}^{{\small \beta}_{2}}=\sqrt{7},\text{ }k_{2}^{{\small \beta}_{2}}%
=4\sqrt{\frac{3}{7}},\text{ }k_{3}^{{\small \beta}_{2}}=\frac{1}{\sqrt{7}}.
\end{array}
\right\}  \label{ex2}%
\end{equation}
From (\ref{canaldenk}), the canal hypersurfaces $\mathfrak{C}^{\{3;\lambda
\}}(s,t,w)$ which are formed as the envelope of a family of pseudo
hyperspheres or pseudo hyperbolic hyperspheres in\textit{ }$E_{1}^{4}$
generated by the curve (\ref{ex1}) are%
\begin{align}
&  \mathfrak{C}^{\{3;1\}}(s,t,w)=\label{ex3}\\
&  \left(
\begin{array}
[c]{l}%
\frac{\sqrt{3}}{7}\left(  28s\left(  -1+\cosh t\cosh w\right)  \cosh s+\left(
7+2\sqrt{21}s\cosh w\sinh t+4\sqrt{7}s\sinh w\right)  \right)  \sinh s,\\
\frac{\sqrt{3}}{7}\left(  28s\left(  -1+\cosh t\cosh w\right)  \sinh s+\left(
7+2\sqrt{21}s\cosh w\sinh t+4\sqrt{7}s\sinh w\right)  \right)  \cosh s,\\
-6s\sin s\cosh t\cosh w+2\left(  \cos s+4s\sin s\right)  +\frac{2s}{\sqrt{7}%
}\left(  -2\sqrt{3}\cosh w\sinh t+3\sinh w\right)  \cos s,\\
2s\left(  -4+3\cosh t\cosh w\right)  \cos s+\frac{2}{7}\left(  7-2\sqrt
{21}s\cosh w\sinh t+3\sqrt{7}\sinh w\right)  \sin s
\end{array}
\right) \nonumber
\end{align}
and%
\begin{align}
&  \mathfrak{C}^{\{3;-1\}}(s,t,w)=\label{ex4}\\
&  \left(
\begin{array}
[c]{c}%
4s\left(  \sqrt{3}+\sqrt{5}\cosh t\cosh w\right)  \cosh s+\left(  \sqrt
{3}+2\sqrt{\frac{5}{7}}s\left(  \sqrt{3}\cosh w\cosh t+2\sinh w\right)
\right)  \sinh s,\\
4s\left(  \sqrt{3}+\sqrt{5}\cosh t\cosh w\right)  \sinh s+\left(  \sqrt
{3}+2\sqrt{\frac{5}{7}}s\left(  \sqrt{3}\cosh w\cosh t+2\sinh w\right)
\right)  \cosh s,\\
-2s\left(  4+\sqrt{15}\cosh t\cosh w\right)  \sin s+\frac{2}{7}\left(
7+\sqrt{35}s\left(  -2\cosh w\sinh t+\sqrt{3}\sinh w\right)  \right)  \cos
s,\\
2s\left(  4+\sqrt{15}\cosh t\cosh w\right)  \cos s+\frac{2}{7}\left(
7+\sqrt{35}s\left(  -2\cosh w\sinh t+\sqrt{3}\sinh w\right)  \right)  \sin s
\end{array}
\right)  ,\nonumber
\end{align}
where the radius function has been taken as $r(s)=2s.$ From (\ref{Kgenel}),
(\ref{Hgenel}) and (\ref{aslitoplu}), the Gaussian, mean and principal
curvatures of the canal hypersurfaces $\mathfrak{C}^{\{3;1\}}$ and
$\mathfrak{C}^{\{3;-1\}}$ are obtained as%
\begin{equation}
\left.
\begin{array}
[c]{l}%
K^{\{3;1\}}=-\frac{\left(  \sqrt{21}+14s\cosh w\sinh t\right)  \cosh w\sinh
t}{4s^{2}\left(  \sqrt{3}+2\sqrt{7}s\cosh w\sinh t\right)  ^{2}},\\
H^{\{3;1\}}=-\frac{3+5\sqrt{21}s\cosh w\sinh t+42s^{2}\cosh^{2}w\sinh^{2}%
t}{s\left(  3+2\sqrt{21}s\cosh w\sinh t\right)  ^{2}},\\
{\small \mu}_{1}^{\{3;1\}}{\small =\mu}_{2}^{\{3;1\}}{\small =}\frac{-1}%
{2s},\text{ }{\small \mu}_{3}^{\{3;1\}}{\small =}-\frac{3\left(  \sqrt
{21}+14s\cosh w\sinh t\right)  \cosh w\sinh t}{\left(  3+2\sqrt{21}s\cosh
w\sinh t\right)  ^{2}}%
\end{array}
\right\}  \label{ex5}%
\end{equation}
and%
\begin{equation}
\left.
\begin{array}
[c]{l}%
K^{\{3;-1\}}=\frac{5\left(  -\sqrt{35}+14s\cosh w\sinh t\right)  \cosh w\sinh
t}{4s^{2}\left(  5-2\sqrt{35}s\cosh w\sinh t\right)  ^{2}},\\
H^{\{3;-1\}}=\frac{1}{3}\left(  \frac{1}{s}+\frac{5\left(  -\sqrt{35}+14s\cosh
w\sinh t\right)  \cosh w\sinh t}{\left(  5-2\sqrt{35}s\cosh w\sinh t\right)
^{2}}\right)  ,\\
{\small \mu}_{1}^{\{3;-1\}}{\small =\mu}_{2}^{\{3;-1\}}{\small =}\frac{1}%
{2s},\text{ }{\small \mu}_{3}^{\{3;-1\}}{\small =}\frac{5\left(  -\sqrt
{35}+14s\cosh w\sinh t\right)  \cosh w\sinh t}{\left(  5-2\sqrt{35}s\cosh
w\sinh t\right)  ^{2}},
\end{array}
\right\}  \label{ex6}%
\end{equation}
respectively. In the following figures, one can see the projections of the
canal hypersurfaces (\ref{ex3}) and (\ref{ex4}) for $w=2$ and $r(s)=2s$ into
$x_{2}x_{3}x_{4}$-spaces in (A) and (B), respectively.

\begin{figure}[H]
\centering
\includegraphics[
height=2.8in, width=4.9in
]{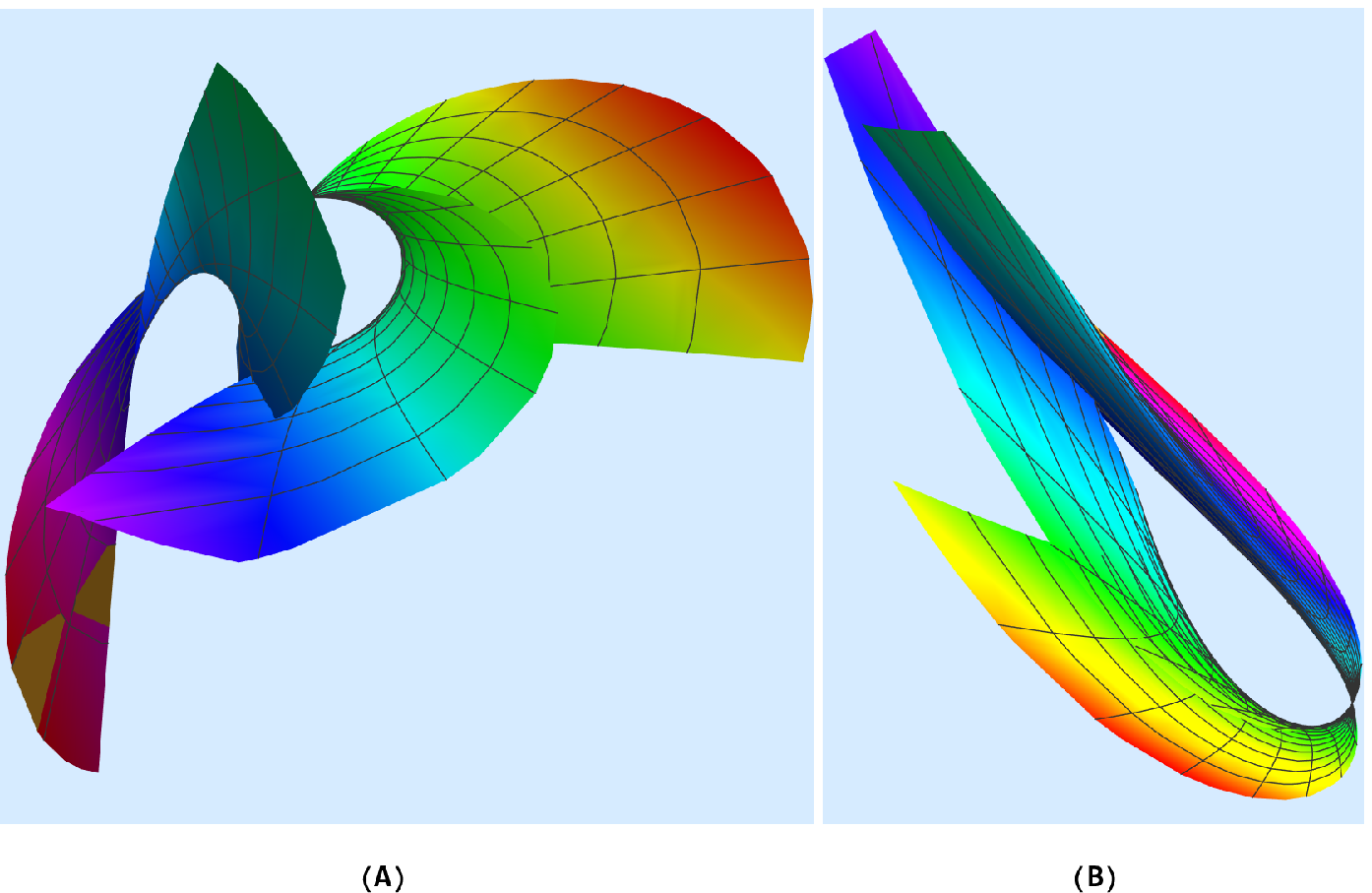}\caption{ }%
\label{fig:1}%
\end{figure}

\end{document}